\documentclass[12pt]{amsart}
\usepackage{latexsym}
\usepackage{amssymb}
\usepackage{amscd}
\usepackage{mathrsfs}
\usepackage{CJK}

\renewcommand\a{\alpha}
\renewcommand\b{\beta}
\newcommand\g{\gamma}
\renewcommand\d{\delta}

\renewcommand\th{\theta}

\newcommand\s{\sigma}

\newcommand\f{\phi}
\newcommand\vf{\varphi}

\renewcommand\r{\rho}

\newcommand\w{\omega}

\newcommand\vS{\varSigma}

\newcommand\vD{\varDelta}

\newcommand\ve{\varepsilon}

\newcommand\BA{\mathbf A}

\newcommand\BQ{\mathbf Q}
\newcommand\BF{\mathbf F}

\newcommand\BZ{\mathbf Z}

\newcommand\BN{\mathbf N}

\newcommand\bB{\mathbf B}

\newcommand\BU{\mathbf U}
\newcommand\BV{\mathbf V}

\newcommand\Bc{\mathbf c}

\newcommand\Bh{\mathbf h}

\newcommand\Bd{\mathbf d}

\newcommand\SC{\mathscr{C}}

\newcommand\SL{\mathscr{L}}

\newcommand\SP{\mathscr{P}}
\newcommand\SQ{\mathscr{Q}}
\newcommand\SH{\mathscr{H}}

\newcommand\SX{\mathscr{X}}

\newcommand\Fg{\mathfrak g}

\newcommand\iv{^{-1}}
\newcommand\wh{\widehat}
\newcommand\wt{\widetilde}

\newcommand\ck{^{\vee}}

\newcommand\ol{\overline}
\newcommand\ul{\underline}
\newcommand\scup{\sqcup}

\newcommand\lra{\leftrightarrow}

\newcommand\lv{\prec}

\newcommand\Ker{\operatorname{Ker}}

\newcommand\rk{\operatorname{rank}\,}

\renewcommand\Im{\operatorname{Im}}

\newcommand\cl{\operatorname{cl}}
\newcommand\weit{\operatorname{wt}}

\newcommand\re{\operatorname{re}}
\newcommand\im{\operatorname{im}}

\newcommand{\isom}{\,\raise2pt\hbox{$\underrightarrow{\sim}$}\,}
\numberwithin{equation}{section}

\newtheorem{thm}{Theorem}[section]
\newtheorem{lem}[thm]{Lemma}
\newtheorem{cor}[thm]{Corollary}
\newtheorem{prop}[thm]{Proposition}

\def \para#1{\par\medskip\textbf{#1}
              \addtocounter{thm}{1}}

\def \remark#1{\par\medskip\noindent
                \textbf{Remark #1}
                \addtocounter{thm}{1}}

\begin{document}
\setlength{\baselineskip}{4.9mm}
\setlength{\abovedisplayskip}{4.5mm}
\setlength{\belowdisplayskip}{4.5mm}
\renewcommand{\theenumi}{\roman{enumi}}
\renewcommand{\labelenumi}{(\theenumi)}
\renewcommand{\thefootnote}{\fnsymbol{footnote}}
\renewcommand{\thefootnote}{\fnsymbol{footnote}}
\allowdisplaybreaks[2]
\parindent=20pt
\medskip
\begin{center}
{\bf Diagram automorphisms and canonical bases \\ 
     for quantum affine algebras } 
\par
\vspace{1cm}
Toshiaki Shoji and Zhiping Zhou 
\\
\vspace{0.7cm}
\title{}
\end{center}

\begin{abstract}
Let $\BU_q^-$ be the negative part of 
the quantum  enveloping algebra $\BU_q(\Fg)$ associated to a simply laced 
Kac-Moody algebra $\Fg$. Let $\Fg^{\s}$ be 
the fixed point subalgebra of $\Fg$ obtained from a diagram automorphism 
$\s$, and put $\ul\BU_q = \BU_q(\Fg^{\s})$. 
Let $\bB$ (resp. $\ul\bB$) be the canonical basis of $\BU_q^-$ 
(resp. $\ul\BU_q^-$). 
Let $\bB^{\s}$ be the set of $\s$-fixed elements in $\bB$.  
Lusztig proved that there exists a natural bijection 
$\bB^{\s} \simeq \ul\bB$ based on his construction of canonical 
bases through the geometric theory of quivers.
In this paper, we prove (the signed bases version of) this fact, 
in the case where $\Fg$ is finite or affine type, 
in an elementary way, in the sense that we don't appeal to
the geometric theory of canonical bases nor Kashiwara's theory 
of crystal bases. 
We also discuss the correspondence between PBW-bases of $\BU_q^-$ and 
of $\ul\BU_q^-$, by using a new type of PBW-bases of $\BU_q^-$ 
obtained by Muthiah-Tingley, 
which is a generalization of the PBW-bases constructed 
by Beck-Nakanjima.  
\end{abstract}

\maketitle
\markboth{SHOJI - ZHIPING} {Diagram automorphisms }
\pagestyle{myheadings}


\par\bigskip 


\begin{center}
{\sc Introduction}
\end{center}

\par\medskip
Let $X$ be a symply laced Cartan datum, and 
$\Fg$ the Kac-Moody algebra associated to $X$.
Let $\BU_q^-$ be the negative part 
of the quantized enveloping algebra $\BU_q(\Fg)$ over $\BQ(q)$
associated to $\Fg$, where $q$ is an indeterminate.  
Let $\BA = \BZ[q,q\iv]$ be the subalgebra of $\BQ(q)$, and 
${}_{\BA}\BU_q^-$ Lusztig's integral form of $\BU_q^-$, which 
is an $\BA$-subalgebra of $\BU_q^-$. 
\par
Let $\s : \Fg \to \Fg$ be the automprhism of $\Fg$ 
obtained from a diagram  automorphism $\s$ on $X$. 
The fixed point subalgebra $\Fg^{\s}$ of $\Fg$ 
by $\s$ is the Kac-Moody algebra associated to 
the Cartan datum $\ul X$.  
$\s$ induces the algebra automorphism 
$\s: \BU_q^- \to \BU_q^-$. 
Let $\ul\BU_q^-$ be the negative part of the quantized 
enveloping algebra $\BU_q(\Fg^{\s})$. 
In [L3], Lusztig constructed the canonical basis $\bB$ of $\BU_q^-$ 
and $\ul \bB$ of $\ul\BU_q^-$. 
$\bB$ (resp. $\ul\bB$) gives an $\BA$-basis of ${}_{\BA}\BU_q^-$
(resp. ${}_{\BA}\ul\BU_q^-$). 
$\s$ acts on $\bB$ as a permutation, 
and we denote by $\bB^{\s}$ the set of $\s$-fixed elements in $\bB$.  
Assume that $\s$ is admissible (see 2.1 for the definition). 
Lusztig proved in [L3, Theorem 14.4.9] that there exists a canonical 
bijection $\ul\bB \isom \bB^{\s}$. 
\par
The rough idea of his proof 
is as follows;  let $Q$ be a quiver whose ambient graph is associated to 
$X$.  Lusztig constructed a certain set $\SP_{Q}$ of simple perverse 
sheaves on various representation spaces of $Q$, and defined a category
$\SQ_{Q}$, whose objects consist of complexes of the form 
$\bigoplus_{A,i}A[i]$, where $A[i]$ is the degree shift by $i \in \BZ$ 
for $A  \in \SP_{Q}$. 
He proved in [L1], [L2] that the Grothendieck group $K(\SQ_Q)$ has a structure
of $\BA$-algebra, and is isomorphic to ${}_{\BA}\BU_q^-$.  Under this 
isomorphism, the set $\SP_Q$ gives the canonical basis $\bB$.
On the other hand, we consider $\s : X \to X$, and choose an orientation of 
$Q$ compatible with $\s$. Then $\s$ induces a functor $\s^* : \SQ_{Q} \to \SQ_{Q}$, 
and one can consider the category $\wt\SQ_Q$, whose objects are the pairs 
$(A, \f)$, where $A \in \SQ_Q$ and $\f : \s^*A \isom A$ satisfies certain 
conditions. He proved in [L3, Theorem 13.2.11] that the Grothendieck group
$K(\wt\SQ_Q)$ is isomorphic to ${}_{\BA}\ul\BU_q^-$. Then 
the set of pairs $(A, \f)$ 
for $A \in \SP_Q$ gives the signed canonical basis 
$\wt{\ul\bB} = \ul\bB \scup -\ul\bB$ of $\ul\BU_q^-$.     
By considering the forgetful functor $(A, \f) \mapsto A$, one obtains a map 
$\wt{\ul\bB} \to \wt\bB$, which induces a bijection $\wt{\ul\bB} \isom \wt\bB^{\s}$ 
(here $\wt\bB$ is the signed canonical basis of $\bB$). 
The ambiguity of the sign is removed by applying the theory of Kashiwara operators, 
the discussion in this part is purely algebraic.
\par
On the other hand, it is known by [GL] that Lusztig's canonical bases 
coincide with Kashiwara's global crystal bases ([K1]). The approach from the
crystal bases theory for the 
proof $\ul\bB \isom  \bB^{\s}$ was done by Naito-Sagaki [NS] and Savage [S]. 
\par
As is seen from the above, the proof $\wt{\ul\bB} \isom \wt\bB^{\s}$ 
heavily depends on the geometric realization of quantized enveloping algebras 
in terms of the theory of perverse sheaves associated to quivers. 
In this paper, in the case where $X$ is affine type or finite type, 
we give an elementary proof for $\wt{\ul\bB} \isom \wt\bB^{\s}$, in the sense
that we don't appeal to the geometric theory of quivers, nor 
Kashiwara's theory of crystal bases. 
The main ingredient is the algebra $\BV_q$ introduced in [SZ].  
Let $\BA' = (\BZ/\ve\BZ)[q,q\iv]$, 
where $\ve$ is the order of $\s$ (here we assume that $\ve = 2$ or 3).   
$\BV_q$ is a quotient algebra of 
${}_{\BA'}\BU_q^{-,\s} = {}_{\BA}\BU_q^-\otimes_{\BA}\BA'$, 
where
$\BU_q^{-,\s}$ is the subalgebra of $\BU_q^-$ 
consisting of $\s$-fixed elements. 
By using the PBW-bases $\SH_{\Bh}$ of $\BU_q^-$ 
constructed in [BN], one can define 
the signed canonical basis $\wt\bB$ of $\BU_q^-$  
(note that this does not work for the non-simply laced case, see 1.9).    
By making use of $\wt\bB$ and PBW-bases $\ul\SX_{\ul\Bh}$ 
on $\ul\BU_q^-$, we can prove 
that $\BV_q \simeq {}_{\BA'}\ul\BU_q^-$ (Theorem 2.4).
Then one can construct the signed canonical basis $\wt{\ul\bB}$ of 
$\ul\BU_q^-$ from those PBW-bases of $\ul\BU_q^-$.  
The above isomorphism induces a surjective homomorphism 
$\f : {}_{\BA'}\BU_q^{-,\s} \to {}_{\BA'}\ul\BU_q^-$. It is proved 
that $\wt{\ul\bB}$ coincides with the image of $\wt\bB^{\s}$ under $\f$,
and that it gives the required bijection (Theorem 3.18).
\par
In the case where $X$ is finite type or affine type with twisted $\ul X$, 
it was proved in [SZ]
that for a given PBW-basis $\ul\SX_{\ul\Bh}$ of $\ul\BU_q^-$, 
one can find a PBW-basis $\SX_{\Bh}$ of $\BU_q^-$ satisfying the property that 
$\s$ acts on $\SX_{\Bh}$ as a permutation, and the set of  
$\s$-fixed elements $\SX_{\Bh}^{\s}$ in $\SX_{\Bh}$ corresponds bijectively to 
$\ul\SX_{\ul\Bh}$ through $\f$, 
which is compatible with the bijection $\wt\bB^{\s} \simeq \wt{\ul\bB}$.    
This fact no longer holds for the affine case with untwistd $\ul X$. 
In this case, the PBW-bases $\SX_{\Bh}$ given by [BN] does not fit
well to our setting.  Instead of those bases, 
by applying the theory of PBW-bases associated to convex orders
due to Muthiah-Tingley [MT], we construct a new type of PBW-basis 
$\SX_{\lv}$ of $\BU_q^-$, starting from the PBW-basis $\ul\SX_{\ul\Bh}$ 
of $\ul\BU_q^-$. We can show that, even though $\s$ does not preserve 
$\SX_{\lv}$, the set $\SX_{\lv}^{\s}$ of $\s$-fixed elements in $\SX_{\lv}$
is in bijection with $\ul\SX_{\ul\Bh}$ through $\f$, which is compatible with 
the bijection $\wt\bB^{\s} \simeq \wt{\ul\bB}$ (Theorem 6.15).
\par
The algebra $\BV_q$ can be defined for the Kac-Moody type $X$  
in general.  In that case, the PBW-bases of $\BU_q^-$ 
are not known.  But as noticed in Remark 2.14, if we assume 
the existence of the canonical basis of $\BU_q^-$ and of $\ul\BU_q^-$
(this certainly holds by Lusztig's result), one can prove 
that $\BV_q \simeq {}_{\BA'}\ul\BU_q$. 
In this way, the bijection $\wt\bB^{\s} \simeq \wt{\ul\bB}$ can be realized 
through an algebra homomorphism. 

\par\medskip
  
\section { Preliminaries }

\para{1.1.}
Let $X = (I, (\ , \ ))$ be a Cartan datum, where 
$(\ ,\ )$ is a symmetric bilinear form 
on $\bigoplus_{i \in I}\BQ\a_i$ (a finite dimensional vector space 
over $\BQ$ with the basis $\{ \a_i\}$ indexed by the vertex set $I$) such that 
$(\a_i, \a_j) \in \BZ$, satisfying the property
\par
\begin{itemize}
\item
$(\a_i,\a_i) \in 2\BZ_{> 0}$ for any $i \in I$, 
\item
$\frac{2(\a_i,\a_j)}{(\a_i,\a_i)} \in \BZ_{\le 0}$ for any $i \ne j$ in $I$.
\end{itemize}

The Cartan datum $X$ is called simply laced if $(\a_i, \a_j) \in \{ 0,-1\}$ for 
any $i \ne j$ in $I$, and $(\a_i,\a_i) = 2$ for any $i \in I$.   
The Cartan datum  $X$ determines a graph with the vertex set $I$.  
Put $a_{ij} = 2(\a_i, \a_j)/(\a_i,\a_i)$ for any $i,j \in I$. 
The matrix $A = (a_{ij})$ is called the Cartan matrix.
If $X$ is simply laced, then $A$ is a symmetric matrix.   
Let $Q = \bigoplus_{i \in I}\BZ\a_i$ be the root lattice, 
and put $Q_+ = \sum_{i \in I}\BN \a_i$. 
\para{1.2.}
In the rest of this paper, we assume that $X$ is of finite type or 
affine type.  However, the discussion for the finite case is
easily obtained from that for the affine case by a suitable modification, 
we mainly consider the affine case, namely, the case where  
the associated graph of $X$ is a Euclidean diagram.  
Let $\Fg$ be the affine Kac-Moody algebra corresponding to the vertex 
set $I = \{ 0,1, \dots, n\}$, and $\Fg_0$ the subalgebra of $\Fg$ of 
finite type corresponding to $I_0 = I - \{ 0\}$. 
(Here we assume that $X$ is irreducible, and not of the type 
$A_{2n}^{(2)}$.  In this case, the vertex 0 is determined uniquely, 
up to the Diagram automorphism.) 
Let $\vD$ be the affine root system  for $\Fg$, and $\vD_0$ 
the root system of $\Fg_0$.
Let $\vD^+$ be the set of positive roots in $\vD$, and $\vD_0^+$ 
the set of positive roots in $\vD_0$. We also denote by 
$\Pi = \{ \a_i \mid i \in I\}$ the set of simple roots in $\vD^+$, 
$\Pi_0 = \{ \a_i \mid i \in I_0\}$ the set of simple roots in $\vD_0^+$.  
Let $\vD^{\re, +}$ be the set of positive real roots, and 
$\vD^{\im, +}$ the set of positive imaginary roots.  Then we have
$\vD^{\re,+} = \vD^{\re,+}_{>} \scup \vD^{\re,+}_{<}$, and 
$\vD^{\im,+} = \BZ_{>0}\d$.  If $X$ is irreducible of type $X_n^{(r)}$ 
in the notation of Kac [Ka], 
$\vD^{\re,+}_{>}$ and $\vD^{\re, +}_{<}$ are given by 
\begin{align*}
\tag{1.2.1}
\vD^{\re,+}_{>} &= \{ \a + m\d \mid \a \in (\vD^+_0)_s, m \in \BZ_{\ge 0} \}
                   \scup \{ \a + mr\d \mid \a \in (\vD^+_0)_l, m \in \BZ_{\ge 0} \}, \\
\vD^{\re,+}_{<} &= \{ \a + m\d \mid \a \in -(\vD^+_0)_s, m \in \BZ_{> 0} \}
                   \scup \{ \a + mr\d \mid \a \in -(\vD^+_0)_l, m \in \BZ_{>  0} \}, \\
\end{align*} 
where $\d$ is the minimal imaginary positive root, and 
$(\vD_0^+)_s$ (resp. $(\vD^+_0)_{l}$) is the set of positive short roots
(resp. positive long roots) in $\vD^+_0$. 
\par
Let $W$ be the Weyl group of $\Fg$ generated by simple reflections 
$\{ s_i \mid i \in I \}$, and $W_0$ the Weyl group of $\Fg_0$ generated by 
$\{ s_i \mid i \in I_0\}$. 

\para{1.3.}
Let $q$ be an indeterminate, and for an integer $n$, 
a positive integer $m $, put 
\begin{equation*}
[n]_q = \frac{q^n - q^{-n}}{q - q\iv}, \quad [m]_q^! = \prod_{i = 1}^m[i]_q, 
         \quad [0]!_q = 1. 
\end{equation*} 
For each $i \in I$, put $q_i = q^{(\a_i,\a_i)/2}$, and consider $[n]_{q_i}$, etc.  
by replacing $q$ by $q_i$ in the above formulas. 
Let $\BU_q^-$ be the negative part of the quantum enveloping algebra 
$\BU_q = \BU_q(\Fg)$  
associated to $\Fg$. Hence $\BU_q^-$ is 
an associative algebra over $\BQ(q)$ with generators $f_i$ ($i \in I$)
satisfying the fundamental relations 
\begin{equation*}
\tag{1.3.1}
\sum_{k = 0}^{1-a_{ij}}(-1)^kf_i^{(k)}f_jf_i^{(1- a_{ij} - k)} = 0
\end{equation*}
for any $i \ne j \in I$, where $f_i^{(n)} = f_i^n/[n]^!_{q_i}$ for 
$n \in \BN$. 
Put $\BA = \BZ[q,q\iv]$ and let ${}_{\BA}\BU_q^-$ be Lusztig's integral
form of $\BU_q^-$, namely the $\BA$-subalgebra of $\BU_q^-$ generated by
$f_i^{(n)}$ for $i \in I$ and $n \in \BN$. 
\par
We define a $\BQ$-algebra involutive automorphism 
$\ol{\phantom{*}} : \BU_q^- \to \BU_q^-$ by $\ol q = q\iv$,
$\ol{f_i} = f_i$ for $i \in I$. 
We define an antil-algebra automorphism  ${}^* : \BU_q^- \to \BU_q^-$
by $f^*_i = f_i$ for any $i \in I$. 

\para{1.4.}
We define a multiplication on $\BU_q^-\otimes \BU_q^-$ by
\begin{equation*}
(x_1\otimes x_2)\cdot (x_1'\otimes x_2') = 
      q^{-(\weit x_2, \weit x_1')}x_1x_1'\otimes x_2x_2',
\end{equation*}
where $x_1, x_1, x_2, x_2'$ are homogeneous elements in $\BU_q^-$ 
under the weight space decomposition $\BU_q^- = \bigoplus_{\nu \in Q}(\BU_q^-)_{\nu}$, 
for $x \in (\BU_q^-)_{\nu}$, here we put $\weit x = \nu$.  
Then $\BU_q^-\otimes \BU_q^-$ turns out to be an associative algebra
with respect to this product. 
One can define a homomorphism $r : \BU_q^- \to \BU_q^-\otimes \BU_q^-$
by $r(f_i) = f_i\otimes 1 + 1 \otimes f_i$. 
It is known that there exists a unique bilinear form $(\ ,\ )$ on 
$\BU_q^-$ satisfying the following properties;

\begin{equation*}
\tag{1.4.1}
\begin{aligned}
(f_i, f_j) &= \d_{ij}(1 - q_i^2)\iv, \\ 
(x, y'y'') &= (r(x), y'\otimes y''), \\ 
(x'x'', y) &= (x'\otimes x'', r(y)),
\end{aligned}
\end{equation*}
where the bilinear form on $\BU_q^-\otimes \BU_q^-$ 
is determined by $(x_1\otimes x_2, x_1'\otimes x_2') = (x_1,x_1')(x_2,x_2')$. 
This bilinear form is symmetric and non-degenerate.

\para{1.5.}
Let $\Bh = ( \dots, i_{-1}, i_0, i_1, \dots)$ 
be a doubly infinite sequence defined in [BN, 3.1] obtained from 
the  element $\xi = \wt\w_1 + \dots + \wt\w_n \in P\ck_{\cl}$, 
where $P\ck_{\cl}$ is the Langlands dual of the weight lattice of 
$\Fg_0$, and $\wt\w_i$ are fundamental coweights in $P\ck_{\cl}$. 
The sequence $\Bh$ satisfies the property 
that $s_{i_p}s_{i_{p+1}}\cdots s_{i_q}$ is a reduced expression in $W$ for 
any $p,q \in \BZ$ such that $p < q$.  
We define $\b_k \in \vD^+$ for $k \in \BZ$ by 
\begin{equation*}
\tag{1.5.1}
\b_k = \begin{cases}
          s_{i_0}s_{i_{-1}}\cdots s_{i_{k + 1}}(\a_{i_k})  &\text{ if $k \le 0$}, \\ 
          s_{i_1}s_{i_2} \cdots s_{i_{k-1}}(\a_{i_k})     &\text{ if $k > 0$}. 
       \end{cases}
\end{equation*}
Then, as in [BN, 3.1], $\b_k$ are all distinct, and 
\begin{align*}
\tag{1.5.2}
\vD^{\re,+}_{>} = \{ \b_k \mid k \in \BZ_{\le 0} \}, \qquad
\vD^{\re,+}_{<} = \{ \b_k \mid k \in \BZ_{> 0}  \}.
\end{align*}

For any $i \in I$, let $T_i : \BU_q \to \BU_q$ be the braid group action. 
For $k \in \BZ, c \in \BN$, we define root vectors $f^{(c)}_{\b_k} \in \BU_q^-$ by 

\begin{equation*}
\tag{1.5.3}
    f^{(c)}_{\b_k} = \begin{cases}
                  T_{i_0}T_{i_{-1}}\cdots T_{i_{k+1}}(f^{(c)}_{i_k}), 
                       &\quad\text{ if } k \le 0, \\
                  T_{i_1}\iv T_{i_2}\iv \cdots T_{i_{k-1}}\iv (f^{(c)}_{i_k}), 
                       &\quad\text{ if } k > 0. 
               \end{cases}  
\end{equation*}
We fix $p \in \BZ$, and let 
$\Bc_{+_p} = (c_p, c_{p-1}, \dots) \in \BN^{\BZ_{\le p}}, 
\Bc_{-_p} = (c_{p+1}, c_{p+2}, \dots) \in \BN^{\BZ_{> p}}$
be functions which are almost everywhere 0. 
We define $L(\Bc_{+_p}), L(\Bc_{-_p}) \in \BU_q^-$ by 
\begin{align*}
\tag{1.5.4}
L(\Bc_{+_p}) &= f_{i_p}^{(c_p)}T_{i_p}(f^{(c_{p-1})}_{i_{p-1}})
        T_{i_p}T_{i_{p-1}}(f^{(c_{p-2})}_{i_{p-2}})  \cdots \\ 
L(\Bc_{-_p}) &= \cdots T_{i_{p+1}}\iv T_{i_{p+2}}\iv(f_{i_{p+3}}^{(c_{p+3})}) 
         T_{i_{p+1}}\iv (f_{i_{p+2}}^{(c_{p+2})})f_{i_{p+1}}^{(c_{p+1})}.
\end{align*}
In the case where $p = 0$, 
we simply write $\Bc_{+_p}, \Bc_{-_p}$ as $\Bc_+, \Bc_-$. 
Thus $(\Bc_{+_p}, \Bc_{-_p})$ is obtained from $(\Bc_+, \Bc_-)$ by the shift by $p$. 
Note that $L(\Bc_+)$ (resp. $L(\Bc_-)$) coincides with 
$f_{\b_0}^{(c_0)}f_{\b_{-1}}^{(c_{-1})}f_{\b_{-2}}^{(c_{-2})}\cdots$ 
(resp. $\cdots f_{\b_3}^{(c_3)}f_{\b_2}^{(c_2)}f_{\b_1}^{(c_1)}$). 
\par
Next we define root vectors for imaginary roots.
For $\a \in \vD_0$, put 
\begin{equation*}
\tag{1.5.5}
d_{\a} = \begin{cases}
          1   &\quad\text{ if $\Fg$ is untwisted, }  \\
          (\a,\a)/2   &\quad\text{ if $\Fg$ is twisted. }
      \end{cases}
\end{equation*}
Put $d_i = d_{\a_i}$ for $i \in I_0$.  
By comparing it with (1.2.1), we see that 
$\a + m\d \in \vD^{\re,+}$ if and only if $d_i \mid m$. 
\par
For $k > 0, i \in I_0$, put
\begin{equation*}
\tag{1.5.6}
\wt\psi_{i, kd_i} = f_{kd_i\d - \a_i}f_i
                 - q^2_if_i f_{kd_i\d - \a_i}.
\end{equation*}

It is known that $\wt\psi_{i,kd_i}$ ($i \in I_0, k \in \BZ_{>0}$) 
are mutually commuting.
For each $i \in I_0, k \in \BZ_{>0}$, we define $\wt P_{i,kd_i} \in \BU_q^-$ 
by the following recursive identity;

\begin{equation*}
\tag{1.5.7}
\wt P_{i,kd_i} = \frac{1}{[k]_{q_i}}\sum_{s = 1}^kq_i^{s-k}\wt\psi_{i,sd_i}\wt P_{i,(k-s)d_i}. 
\end{equation*}  

For a fixed $i \in I_0$, regarding $\wt P_{i,k}$ ($k \in \BZ_{>0}$) as 
elementary symmetric polynomials, we define Schur polynomials by making use of 
the determinant formula; for each partition $\r^{(i)}$, put
\begin{equation*}
\tag{1.5.8}
S_{\r^{(i)}} = \det \bigl(\wt P_{i, (\r'_k - k + m)d_i}\bigr)_{1 \le k, m \le t}
\end{equation*}
where $(\r'_1, \dots, \r'_t)$ is the dual partition of $\r^{(i)}$. 
For an $I_0$-tuple of partitions 
$\Bc_0 = (\r^{(i)})_{i \in I_0}$, we define $S_{\Bc_0}$ by 
\begin{equation*}
\tag{1.5.9}
S_{\Bc_0} = \prod_{i \in I_0}S_{\r^{(i)}}.
\end{equation*}

It is known by [D, 6.3.2] (and [Be] for the untwisted case) that 
\begin{equation*}
\tag{1.5.10}
\begin{aligned}
&T_{i_{p+1}}\iv T_{i_{p+2}}\iv \cdots T_{i_0}\iv (S_{\Bc_0}) \in \BU_q^- 
&\quad\text{ if $p \le 0$ }, \\
&T_{i_p}\cdots T_{i_2}T_{i_1}(S_{\Bc_0}) \in \BU_q^-
&\quad\text{ if $p > 0$. }  
\end{aligned}
\end{equation*}
\par
We denote by $\SC$ the set of triples $\Bc = (\Bc_+, \Bc_0, \Bc_-)$, 
where $\Bc_+ \in \BN^{\BZ_{\le 0}}, \Bc_- \in \BN^{\BZ_{> 0}}$ are 
functions almost everywhere 0,     
and $\Bc_0$ is an $I_0$-tuple of partitions. For each $\Bc \in \SC, 
p \in \BZ$, we define $L(\Bc,p) \in \BU_q^-$ by  
\begin{equation*}
\tag{1.5.11}
L(\Bc, p) = \begin{cases}
         L(\Bc_{+_p})\times \bigl(T\iv_{i_{p +1}}T\iv_{i_{p+2}}\cdots T\iv_{i_0}(S_{\Bc_0})\bigr)
               \times L(\Bc_{-_p}), &\quad\text{ if } p \le 0, \\
         L(\Bc_{+_p}) \times \bigl(T_{i_p}\cdots T_{i_2} T_{i_1} (S_{\Bc_0})\bigr) 
               \times L(\Bc_{-p}),  &\quad\text{ if } p > 0.          
               \end{cases}
\end{equation*}

Note that, by definition, 

\begin{align*}
\tag{1.5.12}
L(\Bc, p-1) &= T_{i_p}\iv L(\Bc, p) \quad\text{ if $p \le 0$ and $c_p = 0$, }  \\
L(\Bc, p+1) &= T_{i_{p+1}} L(\Bc, p)  \quad\text{ if $p > 0$ and $c_{p+1} = 0$.}
\end{align*}

The following results are known.

\begin{thm}[{[BN]}]  
We fix $\Bh$ and $p$ as before.
\begin{enumerate}
\item 
$L(\Bc, p) \in {}_{\BA}\BU_q^-$.
\item
For various $\Bc \in \SC$, $L(\Bc, p)$ are almost orthonormal with 
respect to the bilinear form given in 1.4, namely, 
\begin{equation*}
\tag{1.6.1}
(L(\Bc, p), L(\Bc', p)) \in \d_{\Bc,\Bc'} + (q\BZ[[q]] \cap \BQ(q)). 
\end{equation*}
In particular, for a fixed $\Bh, p$, $\{ L(\Bc, p) \mid \Bc \in \SC \}$ 
gives a basis of $\BU_q^-$.  
\item \
Assume that $\Fg$ is simply laced.  Then 
$\{ L(\Bc, p) \mid \Bc \in \SC\}$ gives an $\BA$-basis of ${}_{\BA}\BU_q^-$.   
\end{enumerate}
\end{thm}

In fact, (i) was proved in Proposition 3.16, 
(ii) in Theorem 3.13 (i), and (iii) in Lemma 3.39 in [BN].  
We write $\SX_{\Bh, p} = \{ L(\Bc, p) \mid \Bc \in \SC \}$, and 
simply write as $\SX_{\Bh} = \SX_{\Bh,p} = \{ L(\Bc, 0) \mid \Bc \in \SC \}$ 
if $p = 0$. 

\para{1.7.}
For each $p \in \BZ$, we define a partial order $\le_p$ on $\SC$ by the 
condition; for $\Bc = (\Bc_+, \Bc_0, \Bc_-), \Bc' = (\Bc'_+, \Bc_0', \Bc'_-) \in \SC$,  
$\Bc <_p \Bc'$ if and only if 
\begin{equation*}
\tag{1.7.1}
\Bc_{+p} \le \Bc'_{+p}  \quad \text{ and }\quad 
       \Bc_{-p} \le \Bc'_{-p} \quad\text{ and one of these are strict,}
\end{equation*}
where both $\le $ are the lexicographic order from left to right, 
e.g., $(c_p, c_{p-1}, \dots) \le (c'_p, c'_{p-1}, \dots)$ 
if there exists $k$ such that 
$c_p = c'_p, \dots, c_{p-k+1} = c'_{p-k+1}$ and that $c_{p-k} < c'_{p-k}$.   
We have the following result.

\begin{prop}[{[BN, Prop. 3.36]}]  
Let $\Bc \in \SC$ and $p \in \BZ$.  Then
\begin{equation*}
\tag{1.8.1}
\ol{L(\Bc, p)} = L(\Bc,p) + \sum_{\Bc <_p \Bd}a_{\Bc, \Bd}L(\Bd, p),
\end{equation*}
where $a_{\Bc, \Bd} \in \BQ(q)$. 
\end{prop}

\para{1.9.}
Now assume that $\Fg$ is simply laced.  Then by Theorem 1.5 (iii), 
$\SX_{\Bh,p}$ gives an $\BA$-basis of ${}_{\BA}\BU_q^-$,  
hence the coefficients $a_{\Bc,\Bd}$ appeared in (1.8.1) 
are contained in $\BA$. 
Then one can construct the canonical basis $\bB_{\Bh, p} = \{ b(\Bc, p) \mid \Bc \in \SC \}$
of $\BU_q^-$, which is characterized by the following properties;
\begin{align*}
\tag{1.9.1}
\ol{b(\Bc,p)} &= b(\Bc, p), \\ 
\tag{1.9.2}
    b(\Bc, p) &= L(\Bc, p) + \sum_{\Bc <_p \Bd}a_{\Bc,\Bd}L(\Bd, p), 
        \quad (a_{\Bc,\Bd} \in q\BZ[q]).
\end{align*}

By the upper triangularity (1.9.2), $\bB_{\Bh, p}$ gives rise to 
an $\BA$-basis of ${}_{\BA}\BU_q^-$, and they are almost orthonormal. 
\para{1.10.}
Let $V$ be a subspace of $\BU_q^-$.  A set $B$ is called 
a signed basis of $V$ if there exists a basis $B'$ of $V$ such that 
$B = B' \cup -B'$.  Similarly, let ${}_{\BA}V$ be an $\BA$-submodule of 
${}_{\BA}\BU_q^-$, a set $B$ is called a signed basis of ${}_{\BA}V$ 
if there exists an $\BA$-basis $B'$ such that $B = B' \cup -B'$.  
\par
The following results are known (the discussion in the proof 
of [L3, Theorem 14.2.3] cn be applied).

\begin{prop} 
Let $\BU_q^-$ be as in 1.9. We define $\SL(\infty)$ as 
\begin{equation*}
\SL(\infty) = \{ x \in \BU_q^- \mid (x,x) \in \BA_0\}, 
\end{equation*}
where $\BA_0 = \BQ[[q]] \cap \BQ(q)$.

\begin{enumerate}
\item $\SL(\infty)$ is an $\BA_0$-submodule of $\BU_q^-$,
where $\{ L(\Bc, p) \mid \Bc \in \SC\}$ gives an $\BA_0$-basis of $\BU_q^-$. 
\item
Let $x$ be an element in  ${}_{\BA}\BU_q^-$ such that $\ol x = x$ and that  
$(x,x) - 1 \in q\BA_0$.  Then there exists $\Bc \in \SC$ such that 
$x \equiv \pm L(\Bc, p) \mod q\SL(\infty)$.  
In particular, $\bB_{\Bh, p} \scup -\bB_{\Bh,p}$ is independent of the choice of 
$\Bh$ and $p$, which is called the signed canonical basis, 
and is denoted by $\wt\bB$. 
\item  By {\rm (ii)}, $*$ preserves $\wt\bB$.  
\end{enumerate}
\end{prop}

\para{1.12.}
We assume that $\BU_q^-$ is of general type.  
Recall the map $r : \BU_q^- \to \BU_q^-\otimes \BU_q^-$ 
as in 1.4.
For any $i \in I$, define $\BQ(q)$-linear maps 
${}_ir, r_i : \BU_q^- \to \BU_q^-$ by 

\begin{equation*}
\tag{1.12.1}
r(x) = f_i\otimes {}_ir(x) + \cdots, \quad 
r(x) = r_i(x)\otimes f_i + \cdots,
\end{equation*}
where $x$ is a homogeneous element, and $\cdots$ denotes 
the terms $y \otimes \BU_q^-$ in the first formula, and the terms 
$\BU_q^-\otimes z$ in the second formula, 
where $y, z$ are homogeneous elements with weight not equal to $-\a_i$.  
It follows from the definition that

\begin{equation*}
\tag{1.12.2}
(f_iy, x) = (f_i, f_i)(y, {}_ir(x)), \quad (yf_i,x) = (f_i,f_i)(y, r_i(x)).
\end{equation*}

The following lemma holds. 

\begin{lem}  
$\bigcap_{i \in I}\Ker {}_ir = \BQ(q)1, \quad 
\bigcap_{i \in I}\Ker r_i = \BQ(q)1.$
\end{lem}

The following results are proved by Lusztig.

\begin{thm}[{[L3, Prop. 38.1.6]}]  
For any $i \in I$, put 
\begin{align*}
\tag{1.14.1}
\BU_q^-[i] &= \{ x \in \BU_q^- \mid T_i\iv(x) \in \BU_q^- \},  \\ 
{}^*\BU_q^-[i] &= *(\BU_q^-[i]) = \{ x \in \BU_q^- \mid T_i(x) \in \BU_q^- \}.
\end{align*}
\begin{enumerate}
\item \ $\BU_q^-[i] = \Ker {}_ir, \quad {}^*\BU_q^-[i] = \Ker r_i$. 
\item \ 
$\BU_q^- = \bigoplus_{n \ge 0}f_i^n\BU_q^-[i] 
         = \bigoplus_{n \ge 0}{}^*\BU_q^-[i]f_i^n$ {\rm (}Orthogonal decompositions{\rm)}. 
\end{enumerate}
\end{thm}

\para{1.15.}
We return to the case where $\Fg$ is simply laced.  
For $b \in \wt\bB$ and $i \in I$, define $\ve_i(b), \ve^*_i(b) \in \BN$ by 
the condition 
\begin{equation*}
\tag{1.15.1}
b \in f_i^{\ve_i(b)}\BU_q^- - f_i^{\ve_i(b) +1}\BU_q^-, 
\quad 
b \in \BU_q^-f_i^{\ve^*_i(b)} - \BU_q^-f_i^{\ve^*_i(b) + 1},
\end{equation*}
namely, $\ve^*_i(b) = \ve_i(b^*)$. 
\par
We fix $\Bh$ as in 1.5.   Then by (1.5.2), $i \in I$ appears in 
the sequence $\Bh$.  
It follows from  the construction in [BN], the infinite sequence $\Bh$ is periodic.
Thus one can choose $p \in \BZ_{\le 0}$ so that $i = i_p$. 
We can find $b(\Bc, p) \in \bB_{\Bh,p}$ such that $b = \pm b(\Bc, p)$, 
where $\Bc = (\Bc_{+p}, \Bc_0, \Bc_{-p}) \in \SC$.  
By (1.9.2), $b(\Bc, p)$ can be written as  a linear combination of 
$L(\Bd, p)$ for various $\Bd = (\Bd_{+p}, \Bd_0, \Bd_{-p})$ such that 
$\Bc \le_p \Bd$.  This implies that $c_p \le d_p$.   
We note that $L(\Bd, p) \in f_{i_p}^{d_p}\BU_q^-[i_p]$ since 
$L(\Bd', p) \in \BU_q^-[i_p]$ if $d'_p = 0$ by (1.5.12) and (1.14.1). 
Then by using Theorem 1.14 (ii), we see that 
\par\medskip\noindent
(1.15.2) \  Let $b = \pm b(\Bc, p)$.  
Assume that $i = i_p$ for $p \le 0$. 
Then  $\ve_i(b) = c_p$.
\par\medskip
The following results are known by Lusztig [L3, Theorem. 14.3.2], 
where he proved them for $\BU_q^-$ of Kac-Moody type, by applying 
the geometric theory of quivers. 
We obtain those results directly from (1.15.2), 
by applying the theory of PBW-bases.

\begin{thm}  
Let $i \in I$ and $a \in \BN$.
\begin{enumerate}  
\item $\{ b \in \wt\bB \mid \ve_i(b) \ge a \}$ 
gives a signed basis of the $\BQ(q)$-vector space $f_i^a\BU_q^-$, and 
a signed basis of the $\BA$-module  
$\sum_{a' \ge a}f_i^{(a')}{}_{\BA}\BU_q^-$. 
\item
$\{ b \in \wt\bB \mid \ve^*_i(b) \ge a \}$ gives a signed basis of 
$\BQ(q)$-vector space $\BU_q^-f_i^a$, and a signed basis of 
the $\BA$-module $\sum_{a' \ge a}{}_{\BA}\BU_q^-f_i^{(a')}$. 
\item
For $b \in \wt\bB$ such that $\ve_i(b) = 0$, there exists a unique 
$b' \in \wt\bB$ such that 
\begin{equation*}
\tag{1.16.1}
f_i^{(a)}b \equiv b' \mod f_i^{a+1}\BU_q.
\end{equation*}
Then $\ve_i(b') = a$.  The map $b \mapsto b'$ gives a bijective 
correspondence
\begin{equation*}
\tag{1.16.2}
\{ b \in \wt\bB \mid \ve_i(b) = 0 \}  \isom \{ b' \in \wt\bB \mid \ve_i(b') = a \}. 
\end{equation*}
\end{enumerate}
\end{thm}

\remark{1.17.}   
We write $f_i^a\BU_q^- \cap {}_{\BA}\BU_q^- = {}_{\BA}(f_i^a\BU_q^-)$.  
Then we have a relation 
\begin{equation*}
\tag{1.17.1}
{}_{\BA}(f_i^a\BU_q^-) = \sum_{a' \ge a}f_i^{(a')}{}_{\BA}\BU_q^-.
\end{equation*}
In fact, if $x \in {}_{\BA}(f_i^a\BU_q^-)$, $x$ can be written as a linear 
combination of $L(\Bc, p) \in \SX_{\Bh, p}$ with coefficients in $\BA$,   
where we may assume that $i = i_p$ for $p \le 0$.  Then $c_i \ge a$ by 
the discussion in 1.15.  Hence $x$ is contained in the right hand side of (1.17.1).
The opposite inclusion is obvious.  Thus (1.17.1) holds. 
Now by using (1.17.1), the condition (1.16.1) can be rewritten as 
the condition 
\begin{equation*}
\tag{1.17.2}
f_i^{(a)}b \equiv b' \mod \sum_{a' \ge a}f_i^{(a')}{}_{\BA}\BU_q^-.
\end{equation*}

\para{1.18.}
By using the bijection (1.16.2), 
we define a map $F_i$ as a composite of the bijections

\begin{equation*}
F_i : \{ b \in \wt\bB \mid \ve_i(b) = a \} \isom \{ b' \in \wt\bB \mid \ve(b') = 0 \}
     \isom  \{ b'' \in \wt\bB \mid \ve(b'') = a + 1 \}.
\end{equation*}

In the case where $\ve_i(b) > 0$, we define $E_i$ as the inverse of the map $F_i$, 
and put $E_i(b) = 0$ if $\ve(b) = 0$.  
The maps $E_i, F_i : \wt\bB  \to \wt\bB \cup \{ 0 \}$ are called the Kashiwara operators.
The following result holds.

\begin{thm} 
For any $b \in \wt\bB$, there exists a sequence 
$i_1, i_2, \dots, i_N \in I$ and $c_1, c_2, \dots, c_N \in \BZ_{> 0}$
such that $b = \pm F_{i_1}^{c_1}\cdots F_{i_N}^{c_N}1$. 
\end{thm}

\begin{proof}
Take $b \in \wt\bB$.  We may assume that $b \ne \pm 1$.  
Then by Lemma 1.13, there exists $i \in I$ such that $b \notin \Ker {}_ir$.  
By Theorem 1.14 (i), $b \notin \BU_q^-[i]$.  
This implies that $\ve_i(b) > 0$ by Theorem 1 14 (ii).  
Then $b' = E_i(b)$ satisfies the condition that $\ve(b') > \ve(b)$.  
Thus the theorem follows by induction on the weight of $b$. 
\end{proof}

\para{1.20.}
For any $i \in I$, $a \in \BN$, put $\wt\bB_{i;a} = \{ b \in \wt\bB \mid \ve_i(b) = a \}$. 
Following Lusztig [L3, 14.4], we shall construct the canonical basis $\bB$ of $\BU_q^-$.
By the induction on the partial order on the weights $\nu \in -Q_+$, we define 
$\bB(\nu)$ as follows; 
\begin{equation*}
\tag{1.20.1}
  \bB(\nu) = \bigcup_{i \in I, a> 0}F_i^a(\bB(\nu + a\a_i) \cap \wt\bB_{i;0}). 
\end{equation*} 
For $\nu = 0$, put $\bB(\nu) = \{ 1 \}$.
Put $\bB = \scup_{\nu \in -Q_+}\bB(\nu) \subset \wt\bB$.  

\begin{thm}[{Lusztig [L3, Thm. 14.4.3]}]
$\bB$ satisfies the following properties. 
\begin{enumerate}
\item  $\wt\bB = \bB \cup -\bB$. 
\item $\bB \cap -\bB = \emptyset$.
\item $\bB$ is an $\BA$-basis of ${}_{\BA}\BU_q^-$, and a basis of 
$\BU_q^-$. 
\end{enumerate}
\end{thm}

\remark{1.22} (i) follows from Theorem 1.19.  The essential difficulty in 
the proof is 
to show that $\bB \cap -\bB = \emptyset$.  This fact was proved by applying 
the theory of Kashiwara operators discussed in [L3, Part III].    
Note that this theory is purely algebraic, and independent from  his geometric 
theory.  But the results are essentially due to Kashiwara [K1]. Lusztig reconstructed
Kashiwara's theory, assuming the existence of the
signed canonical basis at the starting point.
So this theory fits to our situation, but strictly speaking, the proof of
$\bB \cap -\bB = \emptyset$ uses Kashiwara's theory. 
$\bB$ coincides with Kashiwara's global crystal basis. 

\para{1.23.}
Since $\wt\bB = \bB \scup -\bB $,  for each $L(\Bc, p) \in \SX_{\Bh,p}$, 
there exists a unique $b \in \bB$ such that 
$b \equiv \pm L(\Bc, p) \mod q\SL_{\BZ}(\infty)$. 
It was shown that this signature is always 1, by [L4, Proposition 8.2] 
if $L(\Bc,p)$ does not
contain imaginary root vectors, and by [BN, Theorem 3.13] in general. 
They determined the signature for the root vectors corresponding to $S_{\Bc_0}$   
by using the theory of extremal weight modules due to 
Kashiwara [K2].   

\par\bigskip 
\section{ The algebra $\BV_q$ }

\para{2.1.}
Let $X = (I, (\ ,\ ))$ be a simply laced Cartan datum, and 
let $\s : I \to I$ be a permutation such that 
$(\s(\a_i), \s(\a_j)) = (\a_i,\a_j)$ for any $i, j \in I$.
Such a $\s$ is called a diagram automorphism.    
Let $\ul I$ be the set of orbits of $\s$ on $I$. 
We assume that $\s$ is admissible, namely for each orbit $\eta \in \ul I$, 
$(\a_i, \a_j) = 0$ for any $i \ne j$ in $\eta$.  
\par
We define a symmetric bilinear form $(\ ,\ )_1$ on 
$\bigoplus_{\eta \in \ul I}\BQ \a_{\eta}$ by 
\begin{equation*}
(\a_{\eta}, \a_{\eta'})_1 = \begin{cases}
                               2|\eta|  &\quad\text{ if } \eta = \eta', \\
               - |\{ (i, j) \in \eta \times \eta' \mid (\a_i, \a_j) \ne 0\}|
                           &\quad\text{ if } \eta \ne \eta'.
                             \end{cases}                               
\end{equation*}
Then $\ul X = (\ul I, (\ ,\ )_1)$ defines a Cartan datum. 
Let $\ul Q = \bigoplus_{\eta \in \ul I}\BZ \a_{\eta}$ be the 
root lattice, and put $\ul Q_+ = \sum_{\eta \in \ul I}\BN\a_{\eta}$. 
\par
Let $\BU_q$ be the quantum affine algebra associated to $X$ as in 1.3.
Since $X$ is simply laced, $[n]_{q_i} = [n]_q$ for any $i \in I$. 
$\s$ induces an algebra automorphism $\s : \BU_q^- \isom \BU_q^-$ 
by $f_i \to f_{\s(i)}$. 
We denote by $\BU_q^{-,\s}$ the subalgebra of $\BU_q^-$ consisting of $\s$-fixed 
elements. 
$\s$ stabilizes ${}_{\BA}\!\BU_q^-$, and we can define 
${}_{\BA}\!\BU_q^{-,\s} = \BU_q^{-,\s} \cap {}_{\BA}\!\BU_q^-$,  
the subalgebra of ${}_{\BA}\!\BU_q^-$ consisting of $\s$-fixed elements. 
\par
Let $\ul{\BU}_q$ the quantum affine algebra associated to $\ul X$, 
and $\ul{\BU}_q^-$ its negative part, 
namely, $\ul{\BU}_q^-$ is the $\BQ(q)$-algebra generated by 
$\ul f_{\eta}$ with $\eta \in \ul I$ satisfying a similar relation as in (1.3.1).

\para{2.2}
Let $\ve$ be the order of $\s$. We assume that $\ve = 2$ or 3
(note that if $X$ is irreducible, then $\ve = 2,3$ or 4.  We exclude 
the case where $\ve = 4$), and   
$\BF = \BZ/\ve\BZ$ be the finite field of $\ve$-elements. 
Put $\BA' = \BF[q,q\iv]$, and consider the $\BA'$-algebra
\begin{equation*}
\tag{2.2.1}
{}_{\BA'}\BU_q^{-,\s} = {}_{\BA}\BU_q^{-,\s}\otimes_{\BA}\BA'
                      \simeq {}_{\BA}\BU_q^{-,\s}/\ve({}_{\BA}\BU_q^{-,\s}).
\end{equation*}    
Let $J$ be the $\BA'$-submodule of ${}_{\BA'}\BU_q^{-,\s}$ consisting of 
elements of the form $\sum_{0 \le i < \ve}\s^i(x)$ for $x \in {}_{\BA'}\BU_q^-$.  
Then $J$ is a two-sided ideal of ${}_{\BA'}\BU_q^{-,\s}$, and we denote by 
$\BV_q$ the quotient algebra ${}_{\BA'}\BU_q^{-,\s}/J$. 
Let $\pi : {}_{\BA'}\BU_q^{-,\s} \to \BV_q$ be the natural map.  
\par
Let $\ul\BU_q^-$ be as before. 
We can define ${}_{\BA}\ul\BU_q^-$ and ${}_{\BA'}\ul\BU_q^-$ similarly to 
${}_{\BA}\BU_q^-$ and ${}_{\BA'}\BU_q^-$. 

\para{2.3.}
For each $\eta \in \ul I$ and $a \in \BN$, put 
$\wt f_{\eta}^{(a)} = \prod_{i \in \eta}f_i^{(a)}$. 
Since $f_i^{(a)}$ and $f_j^{(a)}$ commute each other for $i,j \in \eta$,  
we have $\wt f^{(a)}_{\eta} \in {}_{\BA}\BU_q^{-,\s}$. We denote its image 
in ${}_{\BA'}\BU_q^{-,\s}$ also by $\wt f^{(a)}_i$.  
Thus we can define $g_{\eta}^{(a)} \in \BV_q$ by 
\begin{equation*}
\tag{2.3.1}
g^{(a)}_{\eta} = \pi(\wt f_{\eta}^{(a)}). 
\end{equation*}
In the case where $a = 1$, we put 
$\wt f_{\eta}^{(1)} = \wt f_{\eta} = \prod_{i \in \eta}f_i$ 
and $g^{(1)}_{\eta} = g_{\eta}$. 
\par
Since the anti-algebra automorphism $*$ commutes with $\s$, $*$ preserves 
${}_{\BA}\BU_q^{-,\s}$, and acts on ${}_{\BA'}\BU_q^{-,\s}$, 
which induces an anti-algebra automorphism  $*$ on $\BV_q$. 
Note that $\wt f_{\eta}^{(a)}$ is $*$-invariant since $f_i$ and $f_j$ commute
for any $i, j \in \eta$. Thus $g^{(a)}_{\eta}$ is $*$-invariant for any $\eta$. 
\par
Recall that ${}_{\BA'}\ul\BU_q^-$ is generated by 
(the image of) $\ul f^{(a)}_{\eta}$ for $\eta \in \ul I$ and $a \in \BN$.
The anti-algebra automorphism $*$ on $_{\BA'}\ul\BU_q^-$ is inherited from 
$*$ on $\ul\BU_q^-$. 
We prove the following result. 

\begin{thm}  
The assignment $\ul f_{\eta}^{(a)} \mapsto g_{\eta}^{(a)}$ gives rise to 
an isomorphism  
$\Phi : {}_{\BA'}\ul\BU^-_q \isom \BV_q$ of $\BA'$-algebras, which is compatible 
with $*$-operations. . 
\end{thm}

\remark{2.5.}
The algebra $\BV_q$ was introduced in [SZ].  The theorem was proved 
in [SZ] in the case where $X$ is finite type or affine type 
subject to the condition that $\s$ preserves $I_0$ (this condition is 
equivalent to that $\ul X$ is twisted type), 
through the precise study of PBW-bases of $\BU_q^-$ and of $\ul\BU_q^-$.
In the present proof, we basically use the canonical basis of $\BU_q^-$, 
and does not use a precise informtion on PBW-bases. 
The discussion here works for any $\s : X \to X$, 
including the case where $X$ is finite type. 
\par
Note that the correspondence for PBW-bases of $\BU_q^-$ and that of 
$\ul\BU_q^-$ under the map $\Phi$ was described in [SZ] in the case 
where $\s$ preserves $I_0$.  But this becomes much more complicated 
if $\s$ does not preserve $I_0$.  We will discuss this case 
later in Section 6, by applying Theorem 2.4. 
\par\medskip

The rest of this section is devoted to the proof of the theorem.
First we note that

\begin{prop}  
The assignment $\ul f_{\eta}^{(a)} \mapsto g^{(a)}_{\eta}$ 
gives a homomorphism $\Phi : {}_{\BA'}\ul\BU_q^- \to \BV_q$.   
\end{prop}

\begin{proof}
The statement in the theorem is also formulated in the case 
where $X$ is of finite type, and the proof of the proposition is 
reduced to a similar property for the case where $X$ is of finite 
type with rank 2.   In [SZ, Proposition 1.10], this was proved 
in the case where $X$ is of finite type.  Thus the proportion holds. 
\end{proof}

\para{2.7.}
Since $\s$ commutes with $r$, the bilinear form $(\ ,\ )$ is $\s$-invariant 
by 1.4,  
namely, $(\s(x), \s(y)) = (x, y)$ for any $x, y \in \BU_q^-$. 
We have
\begin{equation*}
\tag{2.7.1}
(\sum_{0 \le i < \ve}\s^i(x), \sum_{0 \le i < \ve}\s^i(y)) 
           \in \ve(\BZ((q))) \cap \BQ(q) 
\end{equation*}
for any $x, y \in {}_{\BA}\BU_q^-$. 
Let $\BF(q)$ be the rational function field of $q$ over $\BF$ which is 
the quotient field of $\BA'$,  
and consider ${}_{\BF(q)}\BV_q = \BV_q \otimes_{\BA'}\BF(q)$. By (2.7.1), 
$(\ ,\ )$ induces a bilinear form on ${}_{\BF(q)}\BV_q$, which we denote 
also by $(\ ,\ )$. 
\par
Let $\wt\bB$ be the signed canonical bases of $\BU_q^-$. 
By Proposition 1.11 (ii),
$\s$ permutes $\wt\bB$.  We denote by $\wt\bB^{\s}$ the set of 
$\s$-fixed elements in $\wt\bB$.
Since $\wt\bB$ gives an $\BA$-basis of ${}_{\BA}\BU_q^-$, 
${}_{\BA}\BU_q^{-,\s}$ is spanned by $\wt\bB^{\s}$. 
Since $\wt\bB$ is almost orthonormal, 
the image of $\wt\bB^{\s}$ on $\BV_q$ is also almost orthonormal, 
namely satisfies the relation similar to (1.6.1), but by 
replacing $q\BZ[[q]] \cap \BQ(q)$ by $\BF[[q]] \cap \BF(q)$.
In particular, they are linearly independent, and 
$\wt\bB^{\s}$ gives a (signed) $\BA'$-basis of $V_q$.  
We show

\begin{prop}  
\begin{enumerate}
\item 
For any $x, y \in {}_{\BA'}\ul\BU_q^-$, 
$(\Phi(x), \Phi(y)) = (x,y)$.  
\item 
The map $\Phi : {}_{\BA'}\ul\BU_q^- \to \BV_q$ is injective. 
\end{enumerate}
\end{prop}  

\begin{proof}
We consider ${}_{\BF(q)}\ul\BU_q^- = {}_{\BA'}\ul\BU_q^-\otimes_{\BA'}\BF(q)$
and the induced bilinear form $(\ ,\ )$ on it.
Then the map $\Phi : {}_{\BA'}\ul\BU_q^- \to \BV_q$ is extended to the map
${}_{\BF(q)}\ul\BU_q^- \to {}_{\BF(q)}V_q$, which also denote by $\Phi$.
For the proof of (ii), it is enough to show the map on ${}_{\BF(q)}\ul\BU_q^-$ 
is injective. 
Let $\ul\SX_{\ul\Bh}$ be a PBW-basis of $\ul\BU_q^-$ associated to 
a sequence $\ul\Bh$. By Theorem 1.6, any element in $\ul\SX_{\ul\Bh}$ is 
contained in ${}_{\BA}\ul\BU_q^-$, but it is not certain whether $\ul\SX_{\ul\Bh}$
gives an $\BA$-basis of ${}_{\BA}\ul\BU_q^-$. Nevertheless the image of 
$\ul\SX_{\ul\Bh}$ to ${}_{\BF(q)}\ul\BU_q^-$ gives an $\BF(q)$-basis of 
${}_{\BF(q)}\ul\BU_q^-$, and those elements are almost orthonormal over $\BF(q)$.
In particular, the bilinear form $(\ ,\ )$ on ${}_{\BF(q)}\ul\BU_q^-$ is 
non-degenerate.  Now (ii) follows from (i). 
\par 
We show (i). 
Let $r: {}_{\BA}\BU_q^- \to {}_{\BA}\BU_q^-\otimes {}_{\BA}\BU_q^-$ be the map 
induced from the map $r$ defined in 1.4.
We define an action of $\s$ on $\BU_q^-\otimes \BU_q^-$ 
by $\s(x\otimes y) = \s(x)\otimes \s(y)$, which 
preserves the algebra structure, and so
$({}_{\BA}\BU_q^-\otimes {}_{\BA}\BU_q^-)^{\s}$ is 
a subalgebra of $\BU_q^-\otimes \BU_q^-$.  
Thus $r$ induces a homomorphism 
${}_{\BA}\BU_q^{-,\s} \to ({}_{\BA}\BU_q^-\otimes{}_{\BA}\BU_q^-)^{\s}$ of algebras. 
We denote by $J_1$ the $\BA$-submodule of $({}_{\BA}\BU_q^-\otimes {}_{\BA} \BU_q^-)^{\s}$ 
spanned by $\sum_{0 \le i < \ve}\s^i(z)$ for 
$z \in {}_{\BA}\BU_q^-\otimes {}_{\BA}\BU_q^-$.
Then $J_1$ is a two-sided ideal of $({}_{\BA}\BU_q^-\otimes {}_{\BA}\BU_q^-)^{\s}$.
Let $\ul r : {}_{\BA}\ul\BU_q^- \to {}_{\BA}\ul\BU_q^-\otimes {}_{\BA}\ul\BU_q^-$ be 
the map defined as in 1.4, by replacing $\BU_q^-$ by $\ul\BU_q^-$. 
\par
For $x = \ul f_{\eta_1}\cdots \ul f_{\eta_k} \in {}_{\BA}\ul\BU_q^-$, 
put $\wt x = \wt f_{\eta_1}\cdots \wt f_{\eta_k} \in {}_{\BA}\BU_q^{-,\s}$.
Here $\ul r(x) \in {}_{\BA}\ul\BU_q^-\otimes {}_{\BA}\ul\BU_q^-$ 
can be written, uniquely, as an $\BA$-linear 
combination of monomials of the form 
$\ul f_{\eta'_1}\cdots \ul f_{\eta'_s}\otimes \ul f_{\eta''_1}\cdots \ul f_{\eta''_t}$   
We define $\wt{\ul r(x)} \in {}_{\BA}\BU_q^{-,\s}\otimes {}_{\BA}\BU_q^{-,\s}$ 
as a linear combination of monomials of the form 
$\wt f_{\eta'_1}\cdots \wt f_{\eta'_s}\otimes \wt f_{\eta''_1}\cdots \wt f_{\eta''_t}$ 
with the same coefficients.  
We note that 
\begin{equation*}
\tag{2.8.1}
\wt{\ul r(x)} \equiv r(\wt x) \mod J_1.
\end{equation*}

We prove (2.8.1) by induction on $k$.  
We assume that $\ve = 2$, and let $\eta_1 = \{ i,i'\}$.
By noticing that $(\a_i, \a_{i'}) = 0$, we have
\begin{align*}
\tag{2.8.2}
r(\wt f_{\eta_1}) = r(f_if_{i'}) 
 &= (f_i \otimes 1 + 1 \otimes f_i)(f_{i'} \otimes 1 + 1 \otimes f_{i'}) \\
 &= f_if_{i'}\otimes 1 + 1 \otimes f_if_{i'}
       + f_i\otimes f_{i'} + f_{i'}\otimes f_i \\
 &\equiv f_if_{i'}\otimes 1 + 1\otimes f_if_{i'} \mod J_1. 
\end{align*}
Since $f_if_{i'}\otimes 1 + 1 \otimes f_if_{i'} = \wt{\ul r(\ul f_{\eta_1})}$,
(2.8.1) holds for $k = 1$.
Now consider $x = \ul f_{\eta_1}y$ 
with $y = \ul f_{\eta_2}\cdots \ul f_{\eta_{k}}$. 
Then $\wt x = \wt f_{\eta_1}\wt y$, and 
$r(\wt x) = r(\wt f_{\eta_1})r(\wt y)$. 
By induction, we may assume that 
$r(\wt y) \equiv \wt{\ul r(y)} \mod J_1$.  
Since $J_1$ is a two-sided ideal, we have
\begin{align*}
r(\wt x) \equiv \wt{\ul r(\ul f_{\eta_1})}\wt{\ul r(y)}  
          = \wt{\ul r(x)} \mod J_1.
\end{align*}  
Thus (2.8.1) holds. The case where $\ve = 3$ is similar. 
\par
Next we show 
\par\medskip\noindent
(2.8.3) \  For any sequence $\eta_1, \dots, \eta_s$ , 
$\eta'_1, \dots, \eta'_t$ in $\ul I$, we have the identity on $\BF(q)$,

\begin{equation*}
(\ul f_{\eta_1}\cdots\ul f_{\eta_s}, \ul f_{\eta'_1}\cdots \ul f_{\eta'_t})
  = (\wt f_{\eta_1} \cdots \wt f_{\eta_s}, \wt f_{\eta'_1}\cdots \wt f_{\eta'_t})
\end{equation*}

Write $x = \ul f_{\eta_1}\cdots \ul f_{\eta_s}$ and 
$y = \ul f_{\eta'_1}\cdots \ul f_{\eta'_t}$.  Assume that $t \ge 2$, and 
put $y' = \ul f_{\eta'_2}\cdots \ul f_{\eta'_t}$. 
Then by (1.4.1),
\begin{equation*}
(x,y) = (x, \ul f_{\eta'_1}y') = (\ul r(x), \ul f_{\eta'_1}\otimes y')
\end{equation*}
$\ul r(x)$ is written as a linear combination of the form 
$\ul f_{\xi_1}\cdots \ul f_{\xi_k}\otimes \ul f_{\zeta_1}\cdots \ul f_{\zeta_m}$.
Thus $(x,y)$ is written as a linear combination of 
\begin{equation*}
\tag{2.8.4}
(\ul f_{\xi_1}\cdots \ul f_{\xi_k}\otimes \ul f_{\zeta_1}\cdots \ul f_{\zeta_m}, 
              \ul f_{\eta'_1}\otimes y')
        = (\ul f_{\xi_1}\cdots \ul f_{\xi_k}, \ul f_{\eta'_1})
          (\ul f_{\zeta_1}\cdots \ul f_{\zeta_m}, y').
\end{equation*}
By induction on $t$, we have
\begin{align*}
\tag{2.8.5}
(\ul f_{\xi_1}\cdots \ul f_{\xi_k}, \ul f_{\eta'_1}) &=
   (\wt f_{\xi_1}\cdots \wt f_{\xi_k}, \wt f_{\eta'_1}), \\
(\ul f_{\zeta_1}\cdots \ul f_{\zeta_m}, y')  &=
   (\wt f_{\zeta_1}\cdots \wt f_{\zeta_m}, \wt y'), 
\end{align*}
where $\wt y' = \wt f_{\eta'_2}\cdots \wt f_{\eta'_t}$. 

On the other hand, if we put $\wt x = \wt f_{\eta_1}\cdots \wt f_{\eta_s}$, 
$\wt y = \wt f_{\eta'_1}\cdots \wt f_{\eta'_t}$, we have

\begin{equation*}
(\wt x, \wt y) = (\wt x, \wt f_{\eta'_1}\wt y') = (r(\wt x), \wt f_{\eta'_1}\otimes \wt y').
\end{equation*}
By (2.8.1), $r(\wt x) = \wt{\ul r(x)} + z$ with $z \in J_1$. 
Here $\wt{\ul r(x)}$ can be written as a linear combination of 
$\wt f_{\xi_1}\cdots \wt f_{\xi_k}\otimes \wt f_{\zeta_1}\cdots \wt f_{\zeta_m}$
with the same coefficients for $\ul f_{\xi_1}\cdots \ul f_{\xi_k} \otimes 
         \ul f_{\zeta_1}\cdots \ul f_{\zeta_m}$ in $\ul r(x)$. 
Thus $(\wt x, \wt y)$ is a linear combination of 

\begin{equation*}
\tag{2.8.6}
(\wt f_{\xi_1}\cdots \wt f_{\xi_k}\otimes \wt f_{\zeta_1}\cdots \wt f_{\zeta_m}, 
                 \wt f_{\eta_1} \otimes \wt y')
   = (\wt f_{\xi_1}\cdots \wt f_{\xi_k}, \wt f_{\eta_1})
     (\wt f_{\zeta_1}\cdots \wt f_{\zeta_m}, \wt y') 
\end{equation*}
and of $(z, \wt f_{\eta_1}\otimes \wt y')$. 
Here $z \in J_1$ is a linear combination of $\sum_{0 \le i < \ve}\s^i(z_1)$
for some $z_1 \in \BU_q^-\otimes \BU_q^-$. Since 
$\wt f_{\eta'_1}, \wt y' \in \BU_q^{-,\s}$, we have
\begin{equation*}
(\sum_i \s^i(z_1),  \wt f_{\eta'_1}\otimes \wt y') 
       = \ve(z_1, \wt f_{\eta'_1}\otimes \wt y'). 
\end{equation*} 
Thus 
\begin{equation*}
\tag{2.8.7}
(z, \wt f_{\eta_1'}\otimes \wt y') = 0 
\end{equation*}
in $\BF(q)$. 
By comparing (2.8.4) with (2.8.6), together with (2.8.5) and (2.8.7), 
we obtain (2.8.3) if $t \ge 2$.   
\par
Since the form $(\ ,\ )$ is symmetric, the proof of (2.8.3) is reduced 
to the case where $s = t = 1$.  
In this case, assuming that $\ve = 2$, put $\eta = \{i,i'\}$ 
and $\eta' = \{ j, j'\}$.  By (2.8.2) and by (1.4.1) for $\BU_q^-$,  

\begin{align*}
(\wt f_{\eta}, \wt f_{\eta'}) = (f_if_{i'}, f_jf_{j'}) 
                &= (r(f_if_{i'}), f_j \otimes f_{j'})  \\
                &= (f_if_{i'}\otimes 1 + 1 \otimes f_if_{i'} 
                   + f_i \otimes f_{i'} + f_{i'}\otimes f_i,  f_j \otimes f_{j'})  \\
                &= (f_i, f_j)(f_{i'}, f_{j'}) + (f_i, f_{j'})(f_{i'}, f_j) \\
                &= \d_{i,j}(1 - q^2)^{-2} + \d_{i,j'}(1 - q^2)^{-2}. 
\end{align*}

A similar computation works also for the case where $\ve = 3$.
Hence we have

\begin{equation*}
(\wt f_{\eta}, \wt f_{\eta'}) = \begin{cases}
                          (1 - q^2)^{-|\eta|}   &\quad\text{ if $\eta = \eta'$, } \\
                           0                    &\quad\text{ if $\eta \ne \eta'$. }
                              \end{cases}                  
\end{equation*}
On the other hand, by (1.4.1) for $\ul\BU_q^-$, we have

\begin{equation*}
(\ul f_{\eta}, \ul f_{\eta'}) =  \d_{\eta, \eta'}(1 - q_{|\eta|})^{-2}. 
\end{equation*}
Thus (2.8.3) holds for $s = t = 1$.  (2.8.3) is now proved. 
\par
Now for $\ul f_{\eta}^c \in {}_{\BA'}\ul\BU_q^-$, 
$\ul f^{(a)}_{\eta} = ([a]_{q_{\eta}}^!)\iv \ul f^a_{\eta}$ gives a well-defined element 
in ${}_{\BF(q)}\ul\BU_q^-$, and similarly for $\wt f^{(a)}_{\eta}$.   
Note that $[a]^!_{q_{\eta}} \equiv ([a]^!_q)^{|\eta|} \pmod \ve$ for any $a \in \BN$. 
Hence the assertion (i) follows from (2.8.3).   
The proposition is proved. 
\end{proof}

\para{2.9.}
Next we shall show that $\Phi$ is surjective. 
For this, we need some preliminaries. 
Let $E_i, F_i$ be Kashiwara operators on $\BU_q^-$ defined in 1.18.
It is easy to see that 
\begin{equation*}
\s \circ E_i \circ \s\iv = E_{\s(i)}, \qquad \s\circ F_i\circ \s\iv = F_{\s(i)}.
\end{equation*}
Moreover, $E_i$ and $E_j$ commute each other for $i, j \in \eta$, and similarly 
for $F_i$ and $F_j$.  Hence we can define, for $\eta \in \ul I$, the Kashiwara operator
$\wt E_{\eta}$ on $\BU_q^-$ as the product of $E_i$ for $i \in \eta$, 
and similarly for $\wt F_{\eta}$. 
$\wt E_{\eta}, \wt F_{\eta}$ commute with the action of $\s$. 
If $\ve_i(b) = 0$ for $b \in \wt\bB^{\s}$, then $\ve_j(b) = 0$ for any $j \in \eta$.
Then $\wt F^{a}_{\eta}b = b'$ satisfies the condition that $b' \in \wt\bB^{\s}$
and that $\ve_i(b') = a$ for any $i \in \eta$.  Thus by using the bijection (1.16.2), 
we obtain a bijection  
\begin{equation*}
\tag{2.9.1}
\wt F^a_{\eta} : \{ b \in \wt\bB^{\s} \mid \ve_i(b) = 0 \text{ for $i \in \eta$ } \} 
                 \isom 
                 \{ b' \in \wt\bB^{\s} \mid \ve_i(b') = a \text{ for $i \in \eta$ } \}.
\end{equation*}
 
\par
Take $b \in \wt\bB^{\s}$, and assume that $\ve_{i}(b) = a$ for $i \in \eta$.  
We consider $(\wt E_{\eta})^a b = \prod_{j \in \eta}E_j^a b = b'$.
Then  $\ve_j(b') = 0$ for any $j \in \eta$.  Moreover $b' \in \wt\bB^{\s}$ 
by (2.9.1). By applying the formula in (1.17.2) (note that $f_i$ are mutually 
commuting for $i \in \eta$), we see that 
\begin{equation*}
\tag{2.9.2}
\wt f^{(a)}_\eta b' = \prod_{j \in \eta}f_j^{(a)} b' \equiv b \mod Z_{\eta; > a }, 
\end{equation*} 
where $Z_{\eta; > a} = \sum_{j \in \eta} \sum_{a' \ge a+1}f_j^{(a')}{}_{\BA}\BU_q^-$. 
Here $Z_{\eta; > a}$ is $\s$-stable,
and Theorem 1.16 (i) implies that, 
\begin{equation*}
\tag{2.9.3}
\wt\bB_{\eta, > a} = \bigcup_{j \in \eta}\{ b'' \in \wt\bB \mid \ve_j(b'') \ge a  + 1 \} 
\end{equation*}
gives an $\BA$-bases of $Z_{\eta; > a }$. 
In particular, $\s$ gives a permutation of $\wt\bB_{\eta; > a}$.
\par
We can now prove

\begin{prop}  
The map $\Phi : {}_{\BA'}\ul\BU_q^- \to \BV_q$ is surjective. 
\end{prop}

\begin{proof}
We know that the image of $\wt\bB^{\s}$ gives a (signed) bases of $V_q$.
Thus it is enough to show, for each $b \in \wt\bB^{\s}$, that  
\begin{equation*}
\tag{2.10.1}
\pi(b) \in \Im \Phi.
\end{equation*}
Take $b \in \wt\bB^{\s}$.  Let $\nu = \nu(b) \in -Q_+$ be the weight of $b$.
If $\nu = 0$, $b = 1$ and (2.10.1) certainly holds. So, assume that $\nu \ne 0$, 
and by induction, we may assume that $b'$ satisfies (2.10.1) if $\nu(b') > \nu$. 
There exists $i \in I$ such that 
$a = \ve_i(b) > 0$ (see the proof of Theorem 1.19).
We also assume that if $b'$ is such that 
$\nu(b') = \nu$ and that $\ve_i(b') > a$, then $b'$ satisfies (2.10.1).
Let $\eta$ be the $\s$-orbit containing $i$.  
Then $\ve_j(b) = a$ for any $j \in \eta$.  
We consider $(\wt E_{\eta})^a b = \prod_{j \in \eta}E_j^a b = b'$.
Then  $b' \in \wt\bB^{\s}$ by (2.9.1). 
By (2.9.2) there exists $z \in Z^{\s}_{\eta; > a}$ such that 
$b = \wt f_{\eta}^{(a)}b' + z$. 
By induction, $\pi(b') \in \Im \Phi$.  Hence
$\pi(\wt f^{(a)}_{\eta}b') = \Phi(\ul f^{(a)}_{\eta})\pi(b')  \in \Im \Phi$.
On the other hand, by (2.9.3) $z$ is written as an $\BA$-linear
combination of the $\s$-orbits of $b'' \in \wt\bB_{\eta;> a}$.  
If $b''$ is not $\s$-stable, its $\s$-orbit is contained in $J$, and its image on 
$\BV_q$ is zero.  If $\s(b'') = b''$, then $b'' \in \wt\bB^{\s}$, and it satisfies 
the property that 
$\nu (b'') = \nu$ and that $\ve_i(b'') > a$.  Hence by induction hypothesis, 
$\pi(b'') \in \Im \Phi$.  It follows that  $\pi(z) \in \Im \Phi$, and we conclude 
that $\pi(b) \in \Im \Phi$.  
(Note that if there does not exist $b'$ such that $\nu(b') = \nu$ and that 
$\ve_i(\nu)> a$, then $z = 0$, and $b$ satisfies (2.10.1).  Thus the induction 
argument works.)
The proposition is proved. 
\end{proof}

\para{2.11.}
Theorem 2.4 now follows from Propositions 2.6, 2.8 and 2.10.
Note the compatiblity with $*$-operation is clear since 
${}_{\BA'}\ul\BU_q^-$ is generated by $\ul f_{\eta}^{(a)}$, and 
$V_q$ is generated by $g^{(a)}_{\eta}$, both are $*$-stable. 

\para{2.12.}
$\s$ acts on $Q$ by $\a_i \mapsto \a_{\s(i)}$, and 
$Q^{\s} \simeq \ul Q$ under the map $\sum_{i \in \eta} \a_{\s(i)} \mapsto \a_{\eta}$. 
We denote by $\ul\nu \in \ul Q$ the weight corresponding to $\nu \in Q^{\s}$.
If $b \in \wt\bB^{\s}$, the weight $\nu = \nu(b)$ is contained in $Q^{\s}$.  
Let $\wt\bB^{\s}(\nu)$ be the set of $b \in \wt\bB^{\s}$ such that $\nu(b) = \nu$.  
The weight space decomposition of $\BU_q^-$ induces a weight space decomposition 
$\BV_q = \bigoplus_{\nu \in Q^{\s}}(\BV_q)_{\nu}$.  We know that 
$\{\pi(b) \mid b \in \wt\bB^{\s} \}$ gives a (signed) $\BA'$-basis of $\BV_q$.  
Then $\{ \pi(b) \mid b \in \wt\bB^{\s}(\nu)\}$ gives a basis of $(\BV_q)_{\nu}$.  
Hence $|\wt\bB^{\s}(\nu)|/2 = \rk (\BV_q)_{\nu}$.   
On the other hand, 
the weight space decomposition of $\ul\BU_q^-$ induces a weight space 
decomposition ${}_{\BA'}\ul\BU_q^- = 
   \bigoplus_{\ul\nu \in \ul Q}({}_{\BA'}\ul\BU_q^-)_{\ul\nu}$. 
The isomorphism $\Phi : {}_{\BA'}\ul\BU_q^- \isom \BV_q$ preserves the 
weight space decomposition, under the identification $Q^{\s} \simeq \ul Q$. 
Hence we have $\dim (\ul\BU_q^-)_{\ul\nu} = \rk (\BV_q)_{\nu}$. 
Thus the following holds. 

\begin{prop}  
For $\nu \in Q^{\s}$, we have 
$|\wt\bB^{\s}(\nu)|/2 = \rk (\BV_q)_{\nu} = \dim (\ul\BU^-_q)_{\ul\nu}$.  
\end{prop}  

\remark{2.14.}
The map $\Phi : {}_{\BA'}\ul\BU_q^- \to \BV_q$ can be 
defined even if $\BU_q = \BU_q(\Fg)$ is the quantum 
enveloping algebra associated to the Kac-Moody algebra $\Fg$
as in Introduction.   In this case, PBW-bases on $\BU_q^-$ 
are not known.  But if we assume the existence of the canonical 
basis of $\BU_q^-$ and of $\ul\BU_q^-$ (this certainly holds by 
Lusztig's results),  by a suitable modifiation of the proof 
of Theorem 2.4, one can show that the isomorphism  
${}_{\BA'}\ul\BU_q^- \isom \BV_q$ still holds for the Kac-Moody case. 

\par\bigskip

\section{ Canonical basis for $\ul\BU_q^-$  }

\para{3.1.}
We keep the notation in Section 2.  
We fix a doubly infinite sequence 
$\Bh = (\dots, \eta_{-1}, \eta_0, \eta_1, \dots)$ of $\ul I$ 
as in 1.5.  By applying the discussion in 1.5 to $\ul\BU_q^-$, 
we can construct the PBW-basis 
$\ul\SX_{\Bh, p} = \{ \ul L(\Bc, p) \mid \Bc \in \ul\SC \}$ of 
$\ul\BU_q^-$, where $\ul\SC$ is defined similarly to $\SC$ in 1.5.
\par 
Let $\vf : {}_{\BA}\ul\BU_q^- \to {}_{\BA}\ul\BU_q^-/\ve({}_{\BA}\ul\BU_q^-) 
   = {}_{\BA'}\ul\BU_q^-$ be the natural map.   
We denote by $\ul L^{\bullet}(\Bc, p) \in {}_{\BA'}\ul\BU_q^-$ 
the image of $\ul L(\Bc,p)$ under the map $\vf$, and put 
$\ul\SX_{\Bh,p}^{\bullet} = \{ \ul L^{\bullet}(\Bc,p) \mid \Bc \in \ul\SC \}$.  
Note that 
$\ul\SX_{\Bh,p} \subset {}_{\BA}\ul\BU_q^-$ by Theorem 1.6 (i), 
but it is not known whether 
$\ul\SX_{\Bh,p}$ gives an $\BA$-basis of ${}_{\BA}\ul\BU_q^-$.  
In the discussion below, by applying the isomorphism ${}_{\BA'}\ul\BU_q^- \isom V_q$, we shall
prove that this certainly holds for any $(\Bh, p)$.
\par
As discussed in Proposition 2.8, we have a bilinear form 
$(\ ,\ )$ on ${}_{\BF(q)}\ul\BU_q^-$ with values in $\BF(q)$, and 
the elements in $\ul\SX_{\Bh, p}^{\bullet}$ are almost orthonormal,
namely 
\begin{equation*}
\tag{3.1.1}
(\ul L^{\bullet}(\Bc,p), \ul L^{\bullet}(\Bc',p)) \in 
        \d_{c,c'} + (q\BF[[q]] \cap \BF(q)).   
\end{equation*}

\par
The bar involution ${}^-$ on ${}_{\BA}\ul\BU_q^-$ induces a bar involution 
on ${}_{\BA'}\ul\BU_q^-$, which we also denote by ${}^-$.  
The map $\Phi$ commutes with bar involutions. 
Since $\Phi$ gives an isomorphism ${}_{\BA'}\ul\BU_q^- \isom \BV_q$
by Theorem 2.4, $\wt{\ul\bB'} = \Phi\iv(\pi(\wt\bB^{\s}))$ gives a signed basis of 
${}_{\BA'}\ul\BU_q^-$.  In the case where $\ve = 2$, $\wt{\ul\bB'}$ gives a unique 
basis $\ul\bB'$.  In the case where $\ve = 3$, we choose a basis $\ul\bB'$ such that 
$\wt{\ul\bB'} = \ul\bB' \cup -\ul\bB'$. Since $\wt\bB$ is almost orthonormal, and 
the bilinear forms on $\BV_q$ and ${}_{\BA'}\ul\BU_q^-$ are compatible with 
$\Phi$ by Proposition 2.8, we see that 

\par\medskip\noindent
(3.1.2) \ $\ul\bB'$ is almost orthonormal, and bar-invariant.
\par\medskip\noindent

We show the following.
\begin{lem}  
The set $\ul\SX^{\bullet}_{\Bh,p} = 
\{ \ul L^{\bullet}(\Bc, p) \mid \Bc \in \ul\SC \}$ gives an $\BA'$-basis 
of ${}_{\BA'}\ul\BU_q^-$. 
\end{lem}  

\begin{proof}
Put $x = \ul L^{\bullet}(\Bc, p)$.  
Since $\ul\bB'$ is an $\BA'$-basis of ${}_{\BA'}\ul\BU_q^-$, 
$x$ can be written as $x = \sum_{b \in \ul\bB'}a_b b$ with $a_b \in \BA'$. 
One can find $t \in \BZ$ such that $a_b = p_bq^t + $ higher terms of $q$
for any $a_b$ appearing in the sum, where $p_b \in \BF$ and $p_b \ne 0$ 
for at least one $b$.
Then by (3.1.2), we have 
\begin{equation*}
(x,x) \equiv q^{2t}\sum_bp_b^2 \mod q^{2t + 1}(\BF[[q]] \cap \BF(q)). 
\end{equation*}  
Since $(x, x) \equiv 1 \mod (q\BZ[[q]] \cap \BF(q))$ by (3.1.1),
this implies that $t = 0$ and $\sum_b p_b^2 = 1$.  
(If $p_b \in \BZ$, we can deduce from this that 
$p_{b_0} = \pm 1$ for a unique $b_0$ and $p_b = 0$ for $b \ne b_0$. 
But this does not hold since $p_b \in \BF$.  We need a further 
consideration.)  
We can write
$x = \sum_bp_bb + y$, where $y$ is a linear combination of 
$b' \in \ul\bB'$ with coefficients in $q\BF[q]$. 
By Proposition 1.8, $\ol x$ can be written as 
$\ol x = x + z$, where $z$ is a linear combination of various 
$\ul L^{\bullet}(\Bd, p)$ for $\Bc <_p \Bd$, with coefficients in $\BF(q)$. 
Hence
\begin{equation*}
\ol x = \sum p_bb + \ol y  = \sum p_bb + y + z, 
\end{equation*}  
and we have $\ol y - y = z$. 
In particular, $z \in {}_{\BA'}\ul\BU_q^-$, and $y$ is uniquely determined 
from $z$, (namely, if we write $z = \sum a_{b}b$ with $a_{b} \in \BA'$, 
then $a_{b}$ is written as $a_{b} = c_{b,+} + c_{b,0} + c_{b,-}$ with 
$c_{b,+} \in q\BF[q], c_{b,-} \in q\iv \BF[q\iv], c_{b,0} \in \BF$.  
Since $a_{b}$ is anti bar invariant, we must have 
$c_{b,0} = 0, c_{b,-} = -\ol c_{b,+}$. 
Then $y = \sum_{b}c_{b,+}b$). 
For any $\Bd$, we denote by $\ul\bB'_{\Bd}$ the subset of $\ul\bB'$ consisting of 
$b$ appearing in the expansion of $\ul L^{\bullet}(\Bd, p)$. 
Put $\ul\bB'_{\Bc,0} = \ul\bB'_{\Bc} - \bigcup_{\Bc <_p \Bd}\ul\bB'_{\Bd}$. 
Thus we ca write 

\begin{equation*}
\tag{3.2.1}
x = \ul L^{\bullet}(\Bc,p) = \sum_{b \in \ul\bB'_{\Bc,0}}p_bb
                         + \sum_{\Bc <_p \Bd}\sum_{b \in \ul\bB'_{\Bd,0}}a_bb,
\end{equation*}
where $p_b \in \BF$ and $a_b \in \BF[q]$. 
We fix a weight $\nu \in \ul Q$, and consider the set $\ul\bB'(\nu)$ of weight 
$\nu$. $\ul\bB'(\nu)$ gives an $\BA'$-basis of $({}_{\BA'}\ul\BU_q^-)_{\nu}$.  
Similarly, we consider the set $\ul\SX(\nu)$ of PBW-basis of 
weight $\nu$ for $\ul\SX = \ul\SX_{\Bh,p}$. $\ul\SX(\nu)$ gives a 
$\BQ(q)$-basis of $(\ul\BU_q^-)_{\nu}$.
We denote by $\ul\SX^{\bullet}(\nu)$ the image of $\ul\SX (\nu)$ 
on ${}_{\BA'}\ul\BU_q^-$.  Then $\ul\SX^{\bullet}(\nu)$ gives an $\BF(q)$-basis
of $({}_{\BF(q)}\ul\BU_q^-)_{\nu}$. 
Note that $\ul\bB'(\nu)$ is a finite set, and it is partitioned as
$\ul\bB'(\nu) = \bigsqcup_{\Bc}\ul\bB'_{\Bc,0}(\nu)$, 
where $\ul\bB'_{\Bc,0}(\nu) = \ul\bB'_{\Bc,0} \cap \ul\bB'(\nu)$.
We fix a total order on $\ul\bB'(\nu)$ compatible with this partition, 
and consider the transition matrix $T$ between $\ul\bB'(\nu)$ and 
$\ul\SX^{\bullet}(\nu)$
with respect to this order. We regard $T$ as a block matrix with respect to this 
partition.  
Then (3.2.1) shows that $T$ is an upper triangular block matrix with coefficients 
in $\BF[q]$, where the diagonal blocks are matrices with coefficients in $\BF$.
In particular, $\det T \in \BF$.  Here $\ul\bB'(\nu)$ and $\ul\SX(\nu)^{\bullet}$ are both
$\BF(q)$-basis of $({}_{\BF(q)}\ul\BU_q^-)_{\nu}$, 
the matrix $T$ is non-singular.  Thus $\det T \in \BF^*$, 
and the inverse matrix $T\iv$ is a matrix with coefficients in $\BF[q]$.   
This shows that $\ul\SX^{\bullet}_{\Bh,p}$ gives an $\BA'$-basis of ${}_{\BA'}\BU_q^-$.
The lemma is proved. 
\end{proof}

As a corollary to Lemma 3.2, we can prove the following result. 

\begin{thm}  
The set $\ul\SX_{\Bh,p} = \{ \ul L(\Bc,p) \mid \Bc \in \ul\SC \}$ 
gives an $A$-basis of ${}_{\BA}\ul\BU_q$. 
\end{thm}

\begin{proof}
The proof is given in a similar way as in [SZ, Cor.1.21].
Let ${}_{\BA}\wh{\ul\BU}_q^-$ be the inverse limit of 
${}_{\BA}\ul\BU_q^-/\ve^n({}_{\BA}\ul\BU_q^-)$.  Then ${}_{\BA}\wh{\ul\BU}_q^-$ 
has a natural structure of the module over 
$\BZ_{\ve}[q,q\iv] = \varprojlim \BA/\ve^n\BA$, 
where $\BZ_{\ve}$ is the ring of $\ve$-adic integers. 
We have a natural embedding ${}_{\BA}\ul\BU_q^- \subset {}_{\BA}\wh{\ul\BU}_q^-$. 
For the proof, it is enough to show that any element $x \in {}_{\BA}\ul\BU_q^-$ 
is written as a linear combination of $\ul\SX_{\Bh,p}$. 
Since $\ul\SX^{\bullet}_{\Bh,p}$ is 
an $\BA'$-basis of ${}_{\BA'}\ul\BU_q^-$ by Lemma 3.2, $x$ can be written as 
a linear combination of $\ul\SX_{\Bh,p}$ with coefficients in 
$\BA$ modulo $\ve({}_{\BA}\ul\BU_q^-)$. 
We regard $x$ as an element in ${}_{\BA}\wh{\ul\BU}_q^-$.  Then 
$x$ can be written as a linear combination of $\ul\SX_{\Bh,p}$ with coefficients 
in $\BZ_{\ve}[q,q\iv]$.   
On the other hand, $x$ is a linear combination of $\ul\SX_{\Bh,p}$ with 
coefficients in $\BQ(q)$. 
Thus those coefficients belong to 
$\BZ_{\ve}[q,q\iv] \cap \BQ(q) = \BZ[q,q\iv]$. The theorem  is proved. 
\end{proof}

\para{3.4.}
Since $\ul\SX_{\Bh,p}$ gives an $\BA$-basis of ${}_{\BA}\ul\BU_q^-$, 
the canonical basis $\ul\bB_{\Bh,p}$ of ${}_{\BA}\ul\BU_q^-$ can be constructed
as in 1.9.  Thus by Proposition 1.11, we obtain the signed canonical basis 
$\wt{\ul\bB} = \ul\bB_{\Bh,p} \cup -\ul\bB_{\Bh,p}$, which is independent from 
the choice of $\Bh$ and $p$. 
For $b \in \wt{\ul\bB}$, and $\eta \in \ul I$, 
$\ve_{\eta}(b), \ve^*_{\eta}(b)$ can be defined similarly to 1.15.
Then Theorem 1.16 holds, by replacing $i \in I, \wt\bB$ by 
$\eta \in \ul I, \wt{\ul\bB}$.    
Now Kashiwara operators $E_{\eta}, F_{\eta} : \wt{\ul\bB} \to \wt{\ul\bB} \cup \{ 0 \}$
can be defined in a similar way as in 1.18.  Theorem 1.19 also holds 
for $\ul\BU_q^-$. 

\para{3.5.}
We shall consider the relationship between $\pi(\wt\bB^{\s})$
and $\Phi(\wt{\ul\bB})$.  For this, we need some preliminary. 
By Theorem 1.14, we have an orthogonal decomposition
$\BU_q^- = \BU_q[i] \oplus f_i\BU_q^-$.
The following results are a $\s$-version of this decomposition.
First we consider the case where $\ve = 2$.  

\begin{lem}  
Assume that $\eta = \{i, j\}$.  Then we have an orthogonal decomposition 
\begin{equation*}
\tag{3.6.1}
\BU_q^- = (\BU_q^-[i] \cap \BU_q^-[j]) \oplus 
             (f_i\BU_q^-[j] \oplus f_j\BU_q^-[i]) \oplus f_if_j\BU_q^-. 
\end{equation*}
\end{lem}

\begin{proof}
First we note that there exists an orthogonal decomposition
\begin{equation*}
\tag{3.6.2}
f_i\BU_q^- =  f_i(\BU_q^-[j] \oplus f_{j}\BU_q^-) 
           =  f_i\BU_q^-[j] \oplus f_if_j\BU_q^-.
\end{equation*}
In fact, since $T_j(f_i) = f_i$, we have 
$f_i(\BU_q^-[j]) \subset \BU_q^-[j]$ by (1.14.1), 
and $f_if_j\BU_q^- \subset f_j\BU_q^-$.
Thus $f_i\BU_q^-[j] \cap f_if_j\BU_q^- = \{ 0\}$.  
(3.6.2) follows from this. 
\par
It follows from (3.6.2) that we have orthogonal decompositions 
\begin{equation*}
\tag{3.6.3}
\begin{aligned}
\BU_q^- &= \BU_q^-[i] \oplus f_i\BU_q^- = \bigl(\BU_q^-[i] \oplus 
                 f_i\BU_q^-[j]\bigr)  \oplus f_if_j\BU_q^-, \\
\BU_q^- &= \BU_q^-[j] \oplus f_j\BU_q^- = 
        \bigl(\BU_q^-[j] \oplus f_j\BU_q^-[i]\bigr) \oplus f_if_j\BU_q^-.
\end{aligned}
\end{equation*}
The second formula is obtained from the first by applying $\s$. 
By comparing two orthogonal decompositions in (3.6.3), we have

\begin{equation*}
\tag{3.6.4}
\BU_q^-[i] \oplus f_i\BU_q^-[j] = \BU_q^-[j] \oplus f_j\BU_q^-[i]
                 \subset \BU_q^-[i] + \BU_q^-[j].
\end{equation*}
We show that 
\begin{equation*}
\tag{3.6.5}
\BU_q^-[i] \oplus f_i\BU_q^-[j] = (\BU_q^-[i] \cap \BU_q^-[j]) \oplus 
            (f_i\BU_q^-[j] \oplus f_j\BU_q^-[i]). 
\end{equation*}
\par
Consider the orthogonal decompositions 
\begin{equation*}
\BU_q^-[i] = f_j\BU_q^-[i] \oplus A, \quad \BU_q^-[j] = f_i\BU_q^-[j] \oplus B
\end{equation*}
with $A \subset \BU_q^-[i]$, $B \subset \BU_q^-[j]$. 
Then we have orthogonal decompositions 

\begin{align*}
\BU_q^-[i] \oplus f_i\BU_q^-[j] &=  f_j\BU_q^-[i] \oplus f_i\BU_q^-[j] \oplus A, \\
\BU_q^i[j] \oplus f_j\BU_q^-[i] &=  f_i\BU_q^-[j] \oplus f_j\BU_q^-[i] \oplus B.  
\end{align*}
By (3.6.4), this implies that $A = B \subset \BU_q^-[i] \cap \BU_q^-[j]$, 
and we have 
\begin{equation*}
f_i\BU_q^-[j] \oplus \BU_q^-[i] \subset 
       f_i\BU_q^-[j] \oplus f_j\BU_q^-[i] \oplus (\BU_q^-[i] \cap \BU_q^-[j]).  
\end{equation*} 
The other inclusion is obvious, hence (3.6.5) holds.
\par
Now (3.6.1) follows from (3.6.3) and (3.6.5). 
The lemma is proved.
\end{proof}

As a corollary, the following generalization is obtained easily. 
\begin{lem}  
Let $\eta = \{i,j\}$.  For any integer $a \ge 0$, we have
\begin{align*}
\tag{3.7.1}
f_i^af_j^a\BU_q^- = f_i^a&f_j^a\bigl(\BU^-_q[i] \cap \BU^-_q[j]\bigr)  \\
           &\oplus \bigl(f_j^{a+1}f_i^a\BU_q^-[i]\oplus f_i^{a+1}f_j^a\BU_q^-[j]\bigr)
           \oplus f_i^{a+1}f_j^{a+1}\BU_q^-. 
\end{align*}
\end{lem}

\para{3.8.}
We shall consider the integral version of (3.7.1).
For any subspace $Z$ of $\BU_q^-$, put ${}_{\BA}Z = Z \cap {}_{\BA}\BU_q^-$.
The following formula holds. 
\begin{equation*}
\tag{3.8.1}
{}_{\BA}(f_i^a\BU_q^-) = {}_{\BA}(f_i^a\BU_q^-[i]) \oplus {}_{\BA}(f_i^{a+1}\BU_q^-).
\end{equation*}
In fact, such a decomposition holds on $\BQ(q)$ by Theorem 1.14.
By (1.17.1) and Theorem 1.16 (i), $\{ b \in \wt\bB \mid \ve_i(b) \ge a \}$
is a (signed) $\BA$-basis of ${}_{\BA}(f_i^a\BU_q^-)$.  
Hence in order to prove (3.8.1), it is enough 
to see that for any $b \in \wt\bB$ such that $\ve_i(b) \ge a$, 
if we write $b = x + y$ with 
$x \in f_i^a\BU_q^-[i], y \in f_i^{a+1}\BU_q^-$, then $x, y  \in {}_{\BA}\BU_q^-$. 
We choose $p \le 0$ so that $i = i_p$, and write $b = \pm b(\Bc, p)$.   
Then 
\begin{equation*}
b(\Bc, p) = L(\Bc, p) + \sum_{\Bc <_p \Bd}a_{\Bc, \Bd}L(\Bd, p).
\end{equation*}
If $\ve_i(b) > a$, then $b(c, p) \in f_i^{a+1}\BU_q^-$, hence $x = 0$,
and $y = \pm b(\Bc, p)$.  Thus the assertion holds.   So assume that
$\ve_i(b) = a$.  In this case, $b(\Bc, p)$ can be decomposed as 
$b(\Bc, p) = x + y$, where 
$x = L(\Bc, p) + \sum_{d_p = a}a_{\Bd}L(\Bd, p)$, and  
$y = \sum_{d_p > a}a_{\Bd}L(\Bd, p)$, where $\Bd_{+p} = (d_p, ...)$.
We have $x \in f_i^a\BU_q^-[i] \cap {}_{\BA}\BU_q^-$, and 
$y \in f^{a+1}_i\BU_q^- \cap {}_{\BA}\BU_q^-$.  Hence (3.8.1) holds. 

By making use of (3.8.1), we obtain the $\BA$-version of the formula (3.6.1).

\begin{lem}  
Let $\eta = \{i,j\}$.  For any integer $a \ge 0$, we have
\begin{align*}
\tag{3.9.1}
{}_{\BA}(f_i^af_j^a\BU_q^-) = \bigl({}_{\BA}&(f_i^af_j^a\BU^-_q[i]) 
                   \cap {}_{\BA}(f_i^af_j^a\BU^-_q[j])\bigr)  \\
           &\oplus \bigl({}_{\BA}(f_j^{a+1}f_i^a\BU_q^-[i])
                      \oplus {}_{\BA}(f_i^{a+1}f_j^a\BU_q^-[j])\bigr)  \\
           &\oplus {}_{\BA}(f_i^{a+1}f_j^{a+1}\BU_q^-), 
\end{align*}
where 
\begin{align*}
\tag{3.9.2}
{}_{\BA}(f_i^af_j^a\BU_q^-) &= \sum_{b \ge a}\sum_{b' \ge a}
                 f_i^{(b)}f_j^{(b')}{}_{\BA}\BU_q^-, \\ 
\tag{3.9.3}
{}_{\BA}(f_i^af_j^a\BU_q^-[i]) &= \sum_{b \ge a}\sum_{b' \ge a}
                 f_i^{(b)}f_j^{(b')}{}_{\BA}(\BU_q^-[i]). \\ 
\end{align*}
\end{lem}

\para{3.10.}
Next we consider the case of $\ve = 3$. 
Assume that $\eta = \{ i,j,k\}$ with $\s : i \mapsto j \mapsto k \mapsto i$. 
By a similar discussion as in the case where $\ve = 2$, we have an orthogonal
decomposition 
\begin{align*}
\tag{3.10.1}
\BU_q^- = (&\BU_q[i] \cap \BU_q[j] \cap \BU_q[k])   \\
        &\oplus f_i(\BU_q^-[j] \cap \BU_q^-[k]) \oplus f_j(\BU_q^-[k] \cap \BU_q^-[i])
         \oplus f_k(\BU_q^-[i] \cap \BU_q^-[j]) \\
         &\oplus (f_if_j\BU_q^-[k] \oplus f_jf_k\BU_q^-[i] \oplus f_kf_i\BU_q^-[j])   \\        
         &\oplus f_if_jf_k\BU_q^-.
 \end{align*}

Thus as in Lemma 3.7 and Lemma 3.9, we have

\begin{lem}  
Assume that $\eta = \{i,j,k\}$ as above. Then for any integer $a \ge 0$, we have
\begin{align*}
\tag{3.11.1}
f_i^af_j^af_k^a\BU_q^- &= f_i^af_j^af_k^a(\BU_q^-[i] \cap \BU_q^-[j] \cap \BU_q^-[k]) \\
         &\oplus \sum_{0 \le s <3}\s^s\bigr(f_i^{a+1}f_j^af_k^a(\BU_q^-[j] 
              \cap \BU_q^-[k])\bigr)   \\
           &\oplus \sum_{0 \le s < 3}\s^s\bigl(f_i^{a+1}f_j^{a+1}f_k^a\BU_q^-[k]\bigr)   \\
           &\oplus f_i^{a+1}f_j^{a+1}f_k^{a+1}\BU_q^-.
\end{align*}
Moreover, the $\BA$-version of (3.11.1) also holds for ${}_{\BA}(f_i^af_j^af_k^a\BU_q^-)$, 
where
\begin{align*}
\tag{3.11.2}
{}_{\BA}(f_i^af_j^af_k^a\BU_q^-) &= \sum_{b \ge a}\sum_{b' \ge a}\sum_{b'' \ge a}
                                    f_i^{(b)}f_j^{(b')}f_k^{(b'')}{}_{\BA}\BU_q^-,  \\ 
\tag{3.11.3}
{}_{\BA}(f_i^af_j^af_k^a\BU_q^-[i]) &= \sum_{b \ge a}\sum_{b' \ge a}\sum_{b'' \ge a}
                                    f_i^{(b)}f_j^{(b')}f_k^{(b'')}{}_{\BA}(\BU_q^-[i]). 
\end{align*}
\end{lem}

\para{3.12.}
Assume that $\eta = \{ i, j\}$.  Then (3.9.1) induces a formula
\begin{align*}
\tag{3.12.1}
{}_{\BA'}(f_i^af_j^a\BU_q^-)^{\s} = \bigl({}_{\BA'}&(f_i^af_j^a\BU^-_q[i]) 
                   \cap {}_{\BA'}(f_i^af_j^a\BU^-_q[j])\bigr)^{\s}  \\
           &\oplus \bigl({}_{\BA'}(f_j^{a+1}f_i^a\BU_q^-[i])
                      \oplus {}_{\BA'}(f_i^{a+1}f_j^a\BU_q^-[j])\bigr)^{\s}  \\
           &\oplus {}_{\BA'}(f_i^{a+1}f_j^{a+1}\BU_q^-)^{\s}. 
\end{align*}
If we consider the image of this formula by $\pi : {}_{\BA'}\BU_q^{-,\s} \to \BV_q$, 
the second term of the right hand side vanishes. Similarly, if we assume that  
$\eta = \{ i,j,k\}$, in the $\BA'$-version of (3.11.), the second term and 
the third term of the right hand side vanish.  
\par
Now for any $\eta \in \ul I$, we define $\BV_q[\eta]$ as the image of 
$\bigl(\bigcap_{i \in \eta}{}_{\BA'}(\BU_q^-[i])\bigr)^{\s}$ under $\pi$, and put 
\begin{align*}
\tag{3.12.2}
{}_{\BA'}(\wt f^a_{\eta}\BV_q[\eta]) = \sum_{a' \ge a}
                 \wt f_{\eta}^{(a')}\BV_q[\eta]. 
\end{align*}
Then by (3.12.1) and by the corresponding formula obtained from (3.11.1), 
we have

\begin{equation*}
\tag{3.12.3}
\sum_{a' \ge a}\wt f_{\eta}^{(a')}\BV_q = {}_{\BA'}(\wt f^a_{\eta}\BV_q[\eta])
                   \oplus \sum _{a'' > a} \wt f^{(a'')}_{\eta}\BV_q. 
\end{equation*}

If we consider $\ul\BU_q^-$, we have a similar decomposition

\begin{equation*}
\tag{3.12.4}
\sum_{a' \ge a}\ul f^{(a')}_{\eta}{}_{\BA'}\ul\BU_q^- = 
      {}_{\BA'}(\ul f^a_{\eta}\ul\BU_q^-[\eta]) \oplus 
             \sum_{a'' > a}\ul f^{(a'')}_{\eta}{}_{\BA'}\ul\BU_q^-. 
\end{equation*}
By comparing (3.12.3) and (3.12.4), we have the following.

\begin{lem}  
The isomorphism $\Phi: {}_{\BA'}\ul\BU_q^- \isom \BV_q$ induces isomorphisms

\begin{align*}
\sum_{a' \ge a}\ul f^{(a')}_{\eta}{}_{\BA'}\ul\BU_q^- &\isom 
                              \sum_{a' \ge a}\wt f_{\eta}^{(a')}\BV_q, \\
{}_{\BA'}(\ul f^a_{\eta}\ul\BU_q^-[\eta]) &\isom 
                              {}_{\BA'}(\wt f^a_{\eta}V_q[\eta]).        
\end{align*}
\end{lem}

\para{3.14.}
We consider the decomposition 
${}_{\BA}(f_i^a\BU_q^-) = {}_{\BA}(f_i^a\BU_q^-[i]) \oplus {}_{\BA}(f_i^{a+1}\BU_q^-)$.  
For any $x \in {}_{\BA}(f_i^a\BU_q^-)$, let $x_{[i]}$ be the projection of $x$
onto ${}_{\BA}(f_i^a\BU_q^-[i])$. 
Take $b \in \wt\bB$ such that 
$\ve_i(b) = 0$.  Let $b' = F_i^ab$ be the unique element in $\wt\bB$ such that
$b' \equiv f_i^{(a)}b \mod f_i^{a+1}\BU_q^-$.  Then $b'$ is written as 
$b' = \pm b(\Bc, p)$ for some $p \le 0$ with $i = i_p$ such that $c_p = a$.  
We write $b(\Bc, p) = L(\Bc, p) + \sum_{\Bc <_p \Bd}a_{\Bd}L(\Bd, p)$ 
with $a_{\Bd} \in q\BZ[q]$. 
By Theorem 1.16, $b(\Bc,p) \in {}_{\BA}(f_i^a\BU_q^-)$. 
It is easy to see that 
\begin{equation*}
\tag{3.14.1}
b(\Bc,p)_{[i]} =  
L(\Bc,p) + \sum_{\Bc <_p \Bd, d_p = a}a_{\Bd}L(\Bd,p).
\end{equation*} 
Note that $f_i^{(a)}b \in {}_{\BA}(f_i^a\BU_q^-)$. 
Since $b' \equiv f_i^{(a)}b \mod f_i^{a+1}\BU_q^-$, we see that 
$b(\Bc,p)_{[i]}$ coincides with $(f_i^{(a)}b)_{[i]}$. 
Let $\SL_{\BZ}(\infty)$ be the $\BZ$-submodule of $\SL(\infty)$ spanned 
by $\wt\bB$, which coincides with the $\BZ$-submodule spanned by 
$\SX_{\Bh,p}$.   
Here note that 
$L(\Bc, p) \equiv b(\Bc, p)_{[i]} \mod q\SL_{\BZ}(\infty)$. 
Thus in the characterization of the canonical basis $b(\Bc, p)$, 
one can replace the role of the PBW-basis $L(\Bc, p)$ by 
$b(\Bc,p)_{[i]} = (f_i^{(a)}b)_{[i]}$, 
in the following way.

\par\medskip\noindent
(3.14.2) \ Let $b \in \wt\bB$ be such that $\ve_i(b) = 0$. 
Then $b' = F_i^ab$ is the unique element in $\BU_q^-$ satisfying the following.

\begin{align*}
\tag{i}
&\ol b' = b', \\
\tag{ii}
&b' \equiv (f_i^{(a)}b)_{[i]} \mod q\SL_{\BZ}(\infty). 
\end{align*}

\para{3.15.}
We consider a generalization of (3.14.2) to the $\s$-setup.
For each $\eta \in \ul I$, we define an $\BA$-submodule of ${}_{\BA}\BU_q^-$
by 
\begin{equation*}
\tag{3.15.1}
{}_{\BA}(f^a_{\eta}\BU_q^-[\eta]) 
       = \bigcap_{i \in \eta}{}_{\BA}(\wt f^a_{\eta}\BU_q^-[i]),
\end{equation*}
where $\wt f_{\eta} = \prod_{i \in \eta}f_i$.   
We consider the decomposition of $\BA$-modules
\begin{equation*}
\tag{3.15.2}
{}_{\BA}(\wt f_{\eta}^a\BU_q^-) = 
    {}_{\BA}(\wt f^{a}_{\eta}\BU_q^-[\eta]) \oplus ... 
\end{equation*}
obtained from (3.9.1) and (3.11.1),   
For each $x \in {}_{\BA}(\wt f_{\eta}^a\BU_q^-)$, we denote by 
$x_{[\eta]}$ the projection of $x$ onto the first factor in (3.15.2).  
\par
Take $b \in \wt\bB^{\s}$ such that $\ve_i(b) = 0$ for any $i \in \eta$, 
and put $b' = \wt F^a_{\eta}b \in \wt\bB^{\s}$.  
Then by (2.9.1), $\ve_i(b') = a$ for 
any $i \in \eta$. We also consider $\wt f^{(a)}_{\eta}b$.  Since 
$\wt f^{(a)}_{\eta}b \in {}_{\BA}(\wt f^a_{\eta}\BU_q^-)$, one can 
define $(\wt f^{(a)}_{\eta}b)_{[\eta]}$. We have the following lemma. 

\begin{lem}  
Take $b \in \wt\bB^{\s}$ such that $\ve_i(b) = 0$ for any $i \in \eta$. 
Put $b'  = \wt F^a_{\eta}b$.  Then $b'$ is the unique element in $\BU_q^-$ 
satisfying the following properties;
\begin{align*}
\tag{i}
\ol b' &= b',  \\
\tag{ii}
b' &\equiv (\wt f_{\eta}^{(a)}b)_{[\eta]}  \mod q\SL_{\BZ}(\infty). 
\end{align*}
\end{lem}  

\begin{proof}
We prove the lemma in the case where $\ve = 2$.  A  similar 
argument works also for the case where $\ve = 3$.
Assume that $\eta = \{ i,j \}$. 
Take $b \in \wt\bB^{\s}$ such that $\ve_i(b) = \ve_j(b) = 0$.
Put $b'' = F^a_ib$, then $b' = F_j^ab''$.  
From (3.14.2), we have

\begin{align*}
\tag{3.16.1}
b'' &\equiv (f_i^{(a)}b)_{[i]} \mod f_i^{a+1}\BU_q^- \cap q\SL_{\BZ}(\infty)  \\
b' &\equiv (f_j^{(a)}b'')_{[j]} \mod f_j^{a+1}\BU_q^- \cap q\SL_{\BZ}(\infty). 
\end{align*}
Thus $b''$ can be written as $b'' = (f_i^{(a)}b)_{[i]} + x$ with 
$x \in q\SL_{\BZ}(\infty)$. 
Here 

\begin{equation*}
\bigl(f_j^{(a)}(f_i^{(a)}b\bigr)_{[i]})_{[j]} \in f_i^af_j^a\BU_q^-[j].
\end{equation*}
In turn, $f_j^{(a)}(f_i^{(a)}b)_{[i]} \in f_i^af_j^a\BU_q^-[i]$. 
Since $f_j(\BU_q^-[i]) \subset \BU_q^-[i]$, we have 
$(\BU_q^-[i])_{[j]} \subset \BU_q^-[i]$. It follows that
\begin{equation*}
\bigl(f_j^{(a)}(f_i^{(a)}b\bigr)_{[i]})_{[j]} \in f_i^af_j^a\BU_q^-[i]
\end{equation*}
and so 
\begin{equation*}
\bigl(f_j^{(a)}(f_i^{(a)}b\bigr)_{[i]})_{[j]} \in f_i^af_j^a(\BU_q^-[i] \cap \BU_q^-[j]). 
\end{equation*}
This implies, by using the decomposition in Lemma 3.9 and (3.16.1), that
\begin{equation*}
\tag{3.16.2}
\bigl(f_j^{(a)}(f_i^{(a)}b\bigr)_{[i]})_{[j]} = (\wt f^{(a)}_{\eta}b)_{[\eta]}.
\end{equation*}
\par
On the other hand, 
we choose $p' \le 0$ such that $i_{p'} = j$, and consider the PBW-basis
$\{ L(\Bc, p') \mid \Bc \in \SC\}$.  $x$ can be written as
$x = \sum_{\Bc}a_{\Bc}L(\Bc, p')$ with 
$a_{\Bc} \in q\BZ[q]$. We have

\begin{align*}
f_j^{(a)}x &= \sum_{\Bc, c_{p'} = 0}f_j^{(a)}a_{\Bc}L(\Bc, p') + 
                \sum_{\Bc, c_{p'} > 0}f_j^{(a)}a_{\Bc}L(\Bc, p'). 
\end{align*} 
The former part of the right hand side of (3.16.3) is contained in $f_j^a\BU_q^-[j]$, 
and the latter part is contained in $f_j^{a+1}\BU_q^-$. 
It follows that 
\begin{equation*}
\tag{3.16.3}
(f_j^{(a)}x)_{[j]} = \sum_{\Bc, c_{p'} = 0}
          a_{\Bc}f_j^{(a)}L(\Bc, p') \in q\SL_{\BZ}(\infty).
\end{equation*}

Summing up the above computation, we have 

\begin{align*}
\tag{3.16.4}
(f_j^{(a)}b'')_{[j]} &= \bigl(f_j^{(a)}(f_i^{(a)}b)_{[i]}\bigr)_{[j]}
                       + (f_j^{(a)}x)_{[j]}  \\
         &\equiv \bigl(f_j^{(a)}(f_i^{(a)}b)_{[i]}\bigr)_{[j]}  \\
         &= (\wt f^{(a)}b)_{[\eta]} \mod q\SL_{\BZ}(\infty).     
\end{align*}
Thus by (3.16.1), we see that 
$b' \equiv (\wt f^{(a)}_{\eta}b)_{[\eta]} \mod q\SL_{\BZ}(\infty)$, 
and $b'$ satisfies th condition (i) and (ii).  It is clear 
that the condition (i) and (ii) determines $b'$ uniquely.  
The lemma is proved.  
\end{proof}

\para{3.17.}
Recall the definition of Kashiwara operators 
$\wt E_{\eta}, \wt F_{\eta} : \wt\bB^{\s} \to \wt\bB^{\s} \cup \{ 0 \}$
on $\BU_q^-$ for $\eta \in \ul I$ in 2.9.  Those operators induce the operators 
on $\pi(\wt\bB^{\s}) \to \pi(\wt\bB^{\s}) \cup \{ 0 \}$, which we denote by 
the same symbol.  On the other hand, the Kashiwara operators 
$E_{\eta}, F_{\eta}$ on $\ul\BU_q^-$ defined in 3.4 induce operators 
$\Phi(\wt{\ul\bB}) \to \Phi(\wt{\ul\bB}) \cup \{ 0 \}$, which we denote 
by the same symbol. 
We are now ready to prove the following result. 

\begin{thm}   
\begin{enumerate}
\item
$\Phi(\wt{\ul\bB}) = \pi(\wt\bB^{\s})$. 
In particular, the natural map

\begin{equation*}
\begin{CD}
{}_{\BA'}\BU_q^{-,\s} @>\pi >>  \BV_q   @>\Phi\iv >> {}_{\BA'}\ul\BU_q^-
\end{CD}
\end{equation*}
gives a bijection $\f : \wt \bB^{\s} \isom \wt{\ul\bB}$, $b \mapsto \Phi\iv(\pi(b))$. 
\item 
The bijection $\f$ is compatible with Kashiwara operators. 
\item 
The bijection $\f$ is compatible with $*$-operation, up to sign. 
\end{enumerate}
\end{thm}   

\begin{proof}
Take $b \in \wt\bB^{\s}$ such that $\ve_i(b) = 0$ for $i \in \eta$, and consider 
$b' = \wt F^a_{\eta}b \in \wt\bB^{\s}$.  We assume that there exists 
$\ul b \in \wt{\ul\bB}$ such that $\Phi(\ul b) = \pi(b)$. 
Put $\ul b' = F^a_{\eta}\ul b \in \wt{\ul\bB}$.  By Theorem 1.19, 
in order to prove (i), 
it is enough to see that $\Phi(\ul b') = \pi(b')$. 
Since $\Phi$ commutes with the action of $\ul f_{\eta}^{(a)}$ and 
that of $\wt f_{\eta}^{(a)}$, we have 
$\Phi(\ul f_{\eta}^{(a)}\ul b) = \pi(\wt f_{\eta}^{(a)}b)$.  
Then by using the isomorphisms in Lemma 3.13, we see that 
\begin{equation*}
\tag{3.18.1}
\pi((\wt f_{\eta}^{(a)}b)_{[\eta]}) = \Phi\bigl((\ul f_{\eta}^{(a)}\ul b)_{[\eta]}\bigr). 
\end{equation*}
\par
Now $b'$ satisfies the properties given in Lemma 3.16.
Then $\pi(b')$ also satisfies the corresponding properties in $\BV_q$, 
namely 
\begin{equation*}
\tag{3.18.2}
\begin{aligned}
 \ol{\pi(b')} &= \pi(b'),  \\
\pi(b') &\equiv \pi\bigl((\wt f_{\eta}^{(a)}b)_{[\eta]}\bigr) \mod q\ul\SL_{\BZ}(\infty), 
\end{aligned}
\end{equation*}
where $\ul\SL_{\BZ}(\infty)$ is the $\BF[q]$-submodule of $\BV_q$ spanned by 
$\pi(\wt\bB^{\s})$, which is also the $\BF[q]$-submodule spanned by 
$\{ \Phi(\ul L(\Bc, p)) \}$.  
\par
By applying (3.14.2) to the case ${}_{\BA'}\ul\BU_q^-$, $\ul b'$ 
satisfies similar properties, hence by applying $\Phi$, $\Phi(\ul b')$
satisfies the properties, 

\begin{equation*}
\tag{3.18.3}
\begin{aligned}
\ol{\Phi(\ul b')} &= \Phi(\ul  b'), \\
\Phi(\ul b') &\equiv \Phi\bigl((\ul f_{\eta}^{(a)}\ul b)_{[\eta]}\bigr)
               \mod q\ul\SL_{\BZ}(\infty). 
\end{aligned}
\end{equation*}

Now the uniqueness property in Lemma 3.16 also holds for $\BV_q$ 
by replacing $q\SL_{\BZ}(\infty)$ by $q\ul\SL_{\BZ}(\infty)$.
Thus by (3.18.1), we conclude that $\pi(b') = \Phi(\ul b')$.  
This proves (i). (ii) also follows from this argument.  
By Theorem 2.4, $\Phi$ is compatible with $*$-operation.  
(iii) follows from this. 
The theorem is proved. 
\end{proof}

\par\bigskip
\section {PBW-bases associated to convex orders }

\para{4.1.}
Let $\BU_q^-$ be the quantum affine algebra of general type. 
In Section 1, we have discussed the PBW-basis 
$\SX_{\Bh,p} = \{ L(\Bc, p) \mid \Bc \in \SC \}$ constructed in [BN], associated to 
$\Bh$ and $p$. 
Note that in the case where $\BU_q^-$ is of finite type, the PBW-basis
can be constructed for any choice of a reduced expression of the longest element 
in the corresponding Weyl group.  In view of this fact, 
the choice of $\Bh$ in the affine case seems to be rather special, where 
a reduced expression of an infinite length of $W$ is fixed. 
In [MT], Muthiah and Tingley constructed the PBW-basis associated to any convex order 
of the affine root system, which includes $\SX_{\Bh, p}$ of [BN] as a special case. 
In this section, we review the results in [MT]. 

\para{4.2}
We follow the notation in Section 1. 
For any real root $\b = \a + m\d \in \vD^{\re}$, we denote by 
$\ol\b = \a \in \vD_0$.
Put $\vD^{\min}_+ = \vD^{\re,+} \scup \{ \d \}$.
For any subset $X$ of $\vD^{\min}_+$, we define a total order 
$\a \lv \b$ on $X$, called a convex order,  by the following two conditions 
(cf. [I2, Definition 3.3]);
\par\medskip
\noindent
(4.2.1) \  
If $\a, \b \in X$ with $\a + \b \in X$, then $\a + \b$ is in between $\a$ and $\b$. 
\par\medskip
\noindent
(4.2.2) \
If $\a \in X, \b \in \vD^{\min}_+ - X$ with $\a + \b \in X$, then 
$\a \lv \a + \b$. 
\par
\medskip
We define a reverse convex order $\lv$ on $X$ by the condition that 
(4.2.1) and (4.2.$2'$), where
\par\medskip\noindent
(4.2.$2'$) \ 
If $\a \in X, \b \in \vD^{\min}_+ - X$ with $\a + \b \in X$, then
$\a + \b \lv \a$.
\par\medskip
The following gives a typical example of the convex order.
The proof is standard. 

\begin{lem}
Let $w = s_{i_1}\cdots s_{i_N}$ be a reduced expression of $w \in W$.
Define $\b_1, \dots, \b_N \in \vD^{\re,+}$ by 
$\b_k = s_{i_1}\cdots s_{i_{k-1}}(\a_{i_k})$
for $k = 1, \dots N$. 
Put $X = \{ \b_1, \b_2, \dots, \b_N \}$, with the order 
$\b_1 \lv \b_2 \lv \cdots \lv \b_N$. 
Then the order $\lv$ is a convex order on $X$.
\end{lem}

\para{4.4.}
We consider an infinite sequence 
$\Bh = (\dots, i_{-2},i_{-1}, i_0, i_1, i_2, \dots)$ of $I$ 
such that any finite segment $i_k, \dots, i_{\ell}$ gives a
reduced expression $s_{i_k}\cdots s_{i_{\ell}}$. 
In that case, we say that $\Bh$ is an infinite reduced word. 
The infinite reduced word is a generalization of $\Bh$ in 1.5.
We define, for any $k \in \BZ$, $\b_k \in \vD^{\re,+}$ by 
\begin{equation*}
\tag{4.4.1}
\b_k = \begin{cases}
               s_{i_0}s_{i_{-1}}\cdots s_{i_{k+1}}(\a_{i_k})
                    &\quad\text{ if } k \le 0, \\
               s_{i_1}s_{i_2}\cdots s_{i_{k-1}}(\a_{i_k}) 
                    &\quad\text{ if } k > 0.

       \end{cases}
\end{equation*} 
Then by Lemma 4.3, 
$\b_0 \lv \b_{-1} \lv \b_{-2} \lv \cdots \lv \b_{-k} \lv \cdots$ 
is a convex order on $\{ \b_k \mid k \le 0\}$, and 
$\cdots \lv \b_k \lv \cdots \lv \b_2 \lv \b_1$ 
is a reverse convex order on $\{ \b_k \mid k > 0\}$.    
\par
We choose $\Bh$ as in [BN] (see 1.5).  Then by (1.5.2), 
$\{ \b _k \mid k \in \BZ \} \scup \{ \d \} = \vD^{\min}_+$, and   
by Lemma 4.3, we see that 
\par\medskip\noindent
(4.4.2) \ The total order on $\vD^{\min}_+$ defined by 
\begin{equation*}
\b_0 \lv \b_{-1} \lv \cdots \lv \d 
           \lv \cdots \lv \b_2 \lv \b_1
\end{equation*}
gives a convex order, which we denote by $\lv_0$. . 

\para{4.5.}
A convex order $\lv$ on $\vD^{\min}_+$ is called a one-row order
if every positive real root is finitely far from one end or the order.
The total order $\lv_0$ in (4.4.2) is a one-row order. 
The following result, which is in some sense the converse of (4.4.2), 
is known by [MT]. 

\begin{prop}[{[MT, Prop. 2.3], [I2]}] 
For a given one-row order $\lv$ on $\vD^{\min}_+$, there exists an 
infinite reduced word $\Bh = (\dots, i_{-1}, i_0, i_1, \dots)$ satisfying 
the following. For any $k \in \BZ$, define $\b_k \in \vD^{\re,+}$ by 
$(4.4.1)$.  
Then the order $\lv$ on $\vD^{\min}_+$ is given by  
$\b_0 \lv \b_{-1} \lv \cdots \lv \d \lv \dots \lv \b_2 \lv \b_1$.  
\end{prop}

\para{4.7.}
Let $\lv$ be a convex order on $\vD_+^{\min}$ such that 
$\b_0$ is minimal, i.e., $\b_0 \lv \b_{-1} \lv \b_{i_-2}\cdots$.
We assume further that $\b_0 = \a_i$ is a simple root, 
Let $s_i = s_{\a_i}$ be the simple reflection.
Then we can define a new convex order $\lv^{s_i}$ on $\vD_+^{\min}$ by 
\begin{equation*}
s_i(\b_{-1}) \lv^{s_i} s_i(\b_{-2}) \lv^{s_i}  \cdots \lv^{s_i} \b_0,
\end{equation*}
where $\b_0$ is maximal in $\lv^{s_i}$. 
Similarly, if $\b_1 = \a_i$ is maximal in $\vD^{\min}_+$, i.e., 
$\cdots \lv \b_2 \lv \b_1$, we can define a new convex order $\lv^{s_i}$
in a similar way, by moving $\b_1$  to the left end. 
If $\b_0 \lv \cdots \lv \d \lv \cdots \lv \b_1$ is a one-row order 
associated the reduced word $\Bh = (\dots, i_{-1}, i_0, i_1, \dots)$, 
then $\b_0 = \a_{i_0}, \b_1 = \a_{i_1}$, and for 
$w = s_{i_0}s_{i_{-1}}\cdots s_{i_k}$, 
$\lv^w = (\cdots ((\lv^{s_{i_0}})^{s_{i_{-1}}})\cdots )^{s_{i_k}}$ 
defines a convex order on $\vD_+^{\min}$.
Similar definition also works for $w = s_{i_1}s_{i_2}\cdots s_{i_k}$. 

\para{4.8.}
We fix a convex order $\lv$ on $\vD^{\min}_+$.
(Note that in this case, the condition (4.2.2) is unnecessary.)  
If $\g \in \vD_0^+$, then $\g, \d - \g \in \vD^{\min}_+$.  
Since $\g + (\d - \g) = \d$, we have $\d- \g \lv \d \lv \g$ or 
$\g \lv \d \lv \d - \g$. In the former case, we have
\begin{equation*}
\d - \g \lv 2\d - \g \lv 3\d - \g \lv \cdots \lv \d
         \lv \cdots \lv 2\d + \g \lv \d + \g \lv \g.
\end{equation*}  
In the latter case, we have
\begin{equation*}
\g \lv \d + \g \lv 2\d + \g \lv \cdots \lv \d
         \lv \cdots \lv 3\d - \g \lv 2\d - \g \lv \d - \g.
\end{equation*}  

Thus for any $\b \in \vD^{\min}_+$ such that $\b \lv \d$, 
the condition whether $\ol\b \in \vD_0^+$ or 
$\ol\b \in -\vD_0^+$ does not depend on the expression 
$\b = m\d + \ol\b$.  We define a subset $Z$ of $\vD_0$ by 
\begin{equation*}
\tag{4.8.1}
Z = \{ \ol\b \in \vD_0 \mid \b \lv \d \}.
\end{equation*}
Then we have $\vD_0 = Z \scup -Z$.  Moreover, $Z$ is a convex set 
in the sense that if $\a, \b \in Z$ with $\a + \b \in \vD_0$, 
then we must have $\a + \b \in Z$. It follows 
that there exists a unique $\ol w \in W_0$ such that 
$Z = \ol w(\vD_0^+)$. The coarse type of the convex order is defined 
by this $\ol w \in W_0$.  If $\ol w = \ol e$ , the identity element in $W_0$, 
this coarse type is called the standard type.
\par
Let $\lv $ be a convex order of standard type.  
Take $\ol w \in W_0$.  Then $\ol w$ acts on the set $\vD^{\min}_+$, and we can define 
a total order $\lv_{\ol w}$ on $\vD^{\min}_+$ by 
the condition $\b \lv \b'$ if and only if $\ol w(\b) \lv_{\ol w} \ol w(\b')$.
Then $\lv_{\ol w}$ is a convex order with coarse type $\ol w$.  

\para{4.9.}
For a fixed one-row order on $\vD^{\min}_+$, we choose 
an infinite reduced word $\Bh = (\dots, i_{-1}, i_0, i_1, \dots)$ as in 
Proposition 4.6. We define root vectors $f_{\b_k}^{(c)}$ by 
\begin{equation*}
\tag{4.9.1}
f_{\b_k}^{(c)} = \begin{cases}
            T_{i_0}T_{i_{-1}}\cdots T_{i_{k+1}}f^{(c)}_{i_k}
                           &\quad\text{ if $k \le 0$, }  \\            
            T\iv_{i_1}T\iv_{i_2}\cdots T\iv_{i_{k-1}}f^{(c)}_{i_k} 
                           &\quad\text{ if $k > 0$. }  
                  \end{cases}
\end{equation*}
In the case where $c = 1$, we simply write them as $f_{\b_k}$. 
By the following remark, $f_{\b_k}$ does not depend on the choice of 
$\Bh$ for a given one-row order (note that $\Bh$ may not be unique 
for a given one-row order).

\begin{prop}[{[MT, Prop. 4.2]}] 
Fix $\b \in \vD^{\min}_+$, and put 
$M_{\b} = \{ (\b_1, \b_2) \in \vD^+ \times \vD^+ \mid \b = \b_1 + \b_2 \}$.
Let $\lv, \lv'$ be two one-row  orders, and 
$f^{\lv}_{\b_k}, f^{\lv'}_{\b_k}$ the corresponding root vectors. 
If the restriction of $\lv, \lv'$  
to all the pairs in $M_{\b}$ agree, then $f^{\lv}_{\b_k} = f^{\lv'}_{\b_k}$. 
\end{prop}

\begin{lem}[{[MT, Lemma 2.8]}]  
Assume given a convex order $\lv$ on $\vD^{\min}_+$.
Then for any finite collection $\b_1, \dots, \b_N$ in $\vD^{\min}_+$, 
there exists a one-row order on $\vD^{\min}_+$ such that 
its restriction to $\{ \b_1, \cdots, \b_N\}$ coincides with $\lv$ on it.
\end{lem}

\para{4.12.}
Fix a convex order $\lv$ on $\vD^{\min}_+$. For a given real root $\b \in \vD^+$, 
let $M_{\b}$ be as before. Note that $|M_{\b}| < \infty$.  Thus by Lemma 4.11, 
one can find a one-row  order on $\vD^{\min}_+$ whose restriction 
to all the pairs in $M_{\b}$ coincides with $\lv$. 
We define a root vector $f_{\b}^{(c)}$ as in (4.9.1) by using this one-row order. 
By Proposition 4.10, $f_{\b}^{(c)}$ does not depend on the choice 
of the one-row order, is determined uniquely by the given convex order $\lv$. 
We write this root vector $f_{\b}$ as $f_{\b}^{\lv}$. 
\par
If $\a_i$ is a simple root, $\a_i$ cannot be written as a sum of 
two positive roots.  Thus for defining $f_{\a_i}^{\lv}$, we can find 
a one-row order such that $\a_i$ is minimal or maximal 
(for example, the infinite reduced word $\Bh$ in [BN] contains 
the term $i = i_p$ for some $p\le 0$. Let $\lv$ be the one-row order 
associated to $\Bh$.  Then $\lv^{w}$ for $w = s_{i_0}\cdots s_{i_p}$ 
is the one-row order whose minimal element is $\a_i$).    
As a corollary to Proposition 4.10, we have 

\begin{cor}[{[MT, Cor. 4.3]}]  
For any simple root $\a_i$ and for any choice of the convex order $\lv$, 
$f^{\lv}_{\a_i} = f_i$.   
\end{cor}

\para{4.14.}
We consider a convex order $\lv$ on $\vD^{\min}_+$ 
with a coarse type $\ol w \in W_0$. 
Thus $\ol w(\vD_0^+)$ is a positive system in $\vD_0$. 
Take $\b \in \vD^{+,\re}$ of the form $\b = n\d \pm \a$, where $n \ge 0$ and 
$\a \in \vD_0$ is a simple root in the positive system $\ol w(\vD_0^+)$.
If $\b$ can be written as $\b = \b_1 + \b_2$ with $\b_1, \b_2 \in \vD^{+,\re}$, 
then $\b_1 \lv \d \lv \b_2$ or $\b_2 \lv \d \lv \b_1$ since $\a$ is simple.  
The case $\b_1 \lv \d \lv \b_2$ occurs if 
$\ol\b_1 \in \ol w(\vD_0^+), \ol\b_2 \in -\ol w(\vD_0^+)$, 
and $\b_2 \lv \d \lv \b_1$ occurs if $\ol\b_2 \in \ol w(\vD_0^+), 
\ol\b_1 \in -\ol w(\vD_0^+)$.  
Thus those two cases are completely determined 
by the coarse type.  In particular, 
the restriction of $M_{\b}$ of $\lv$ is independent of the choice of 
the convex order as far as the coarse type is fixed.  Thus by Proposition 4.10, 
we have   

\begin{cor}[{[MT, Cor. 4.6]}]   
Let $\a \in \vD_0$ be a simple root in the positive system $\ol w(\vD_0^+) \subset \vD_0$.
Then for any $n \ge 0$, $f^{\lv}_{n\d  \pm \a}$ does not depend on the 
choice of $\lv$ with given coarse type $\ol w$. 
\end{cor}

\remark{4.16.} \ 
If $\Bh$ is an infinite reduced word given in [BN], then 
$\{ \b_k \mid k \le 0 \} = \vD^{\re,+}_{>}$ and 
$\{ \b_k \mid k > 0 \} = \vD^{\re,+}_{<}$
(see (1.5.2)).
Thus the coarse type of the convex order $\lv_0$ 
is $\ol e \in W_0$, 
namely $\lv_0$ has the standard coarse type.  If $\lv$ has the standard 
coarse type, by Corollary 4.15, $f_{n\d \pm \a_i}^{\lv}$ coincides with 
$f_{n\d \pm \a_i}$ defined in 1.5. 

\para{4.17.}
We fix a convex order $\lv$ with coarse type $\ol w \in W_0$. 
For each $i \in I_0$, let $\ol\g_i = \ol w\a_i$ be the simple root 
in the positive system $\ol w(\vD^+_0)$, and define $\g_i \in \vD^+$ by 
$\g_i = \ol\g_i$ if $\g_i \in \vD^+_0$, and by $\g_i = \d + \ol\g_i$
if $\ol\g_i \in -\vD^+_0$. (Hence $\g_i$ is the minimal affine root
such that its projection to $\vD_0$ is $\ol\g_i$.) 
Then $f^{\lv}_{k\d \pm \g_i}$ does not depend on $\lv$, only depends on 
the coarse type $\ol w$ by Corollary 4.13, which 
we denote by $f_{k\d \pm \g_i}^{\ol w}$. 
Put
\begin{equation*}
\tag{4.17.1}
\Psi^{\ol w}_{i,kd_i} = f^{\ol w}_{kd_i\d - \g_i}f^{\ol w}_{\g_i} 
            - q_i^2f^{\ol w}_{\g_i}f^{\ol w}_{kd_i\d - \g_i}.
\end{equation*}
By Remark 4.16, $\Psi^{\ol e}_{i, kd_i}$ coincides with $\wt\psi_{i, kd_i}$
defined in (1.5.6). 
Note that $\Psi^{\ol w}_{i, kd_i}$ are mutually commuting by the theorem below. 
Thus $\wt P^{\ol w}_{i,kd_i}$ is defined similarly to (1.5.7) by replacing 
$\wt\psi_{i,kd_i}$ by $\Psi^{\ol w}_{i,kd_i}$.  
We denote by $S_{\r^{(i)}}^{\ol w}$ the function corresponding 
to $S_{\r^{(i)}}$ in (1.5.8), and define $S_{\Bc_0}^{\ol w}$ 
similarly to (1.5.9). 
We put $S_{\Bc_0}^{\lv} = S_{\Bc_0}^{\ol w}$. 
Thus in the case where $\ol w = \ol e$, namely, $\lv$ is of standard type,  
$S_{\Bc_0}^{\lv} = S_{\Bc_0}$.

\begin{thm}[{[MT, Thm. 4.13]}]  
Let $\lv$ be a convex order with coarse type $\ol w$.  
Then $\Psi^{\ol w}_{i,kd_i} = T_w\Psi^{\ol e}_{i,kd_i}$, 
where $w \in W$ is the minimal length lift of $\ol w \in W_0$
to $W$.  
\end{thm}

\para{4.19.}
Let $\Bc = (\Bc_+, \Bc_0, \Bc_-) \in \SC$ be as in 1.5. 
For a convex order $\lv$ and $\Bc \in \SC$, we define an element
\begin{equation*}
\tag{4.19.1}
L(\Bc, \lv) = f_{\b_1}^{\lv,(c_{\b_1})}\cdots f_{\b_N}^{\lv, (c_{\b_N})}
                 S^{\lv}_{\Bc_0}
              f_{\g_M}^{\lv, (c_{\g_M})}\cdots f_{\g_1}^{\lv, (c_{\g_1})},
\end{equation*}
where $\Bc_+ = (c_{\b}), \Bc_- = (c_{\g})$, and  
$\b_1 \lv \cdots \lv \b_N \lv \d \lv \g_M \lv \cdots \lv \g_1$ 
are the real roots $\b,\g$ for which $c_{\b}, c_{\g} \ne 0$.  
It is shown in a similar way as in the proof of Theorem 1.6 (ii) that 
$\{ L(\Bc, \lv) \mid \Bc \in \SC\}$ is almost orthonormal. 
This implies that $\SX_{\lv} = \{ L(\Bc, \lv) \mid \Bc \in \SC\}$ gives 
a basis of $\BU_q^-$
for any convex order $\lv$, which is called the 
$PBW$-basis associated to $\lv$. 
\par
In the case of the convex order $\lv_0$, 
$L(\Bc, \lv_0)$ coincides with $L(\Bc, 0)$ with $p = 0$ in (1.5.11), 
and more generally, $L(\Bc, p)$ in (1.5.11) coincides with 
$L(\Bc, \lv_0^w)$, where $w = s_{i_0}s_{i_{-1}}\cdots s_{i_{p+1}}$ if $p \le 0$
and $w = s_{i_1}s_{i_2}\cdots s_{i_{p-1}}$ if $p > 0$. 
  
The followings are known. 

\begin{prop} [{[MT, Prop. 4.21]}] 
For any convex order $\lv$, $L(\Bc, \lv)$ is contained in ${}_{\BA}\BU_q^-$. 
\end{prop}    

\begin{prop} [{[MT, Prop. 4.27]}]  
For any convex order $\lv$ and for any $\Bc \in \SC$, we have
\begin{equation*}
\ol{L(\Bc, \lv)} = L(\Bc, \lv) + \sum_{\Bc < \Bc'}a_{\Bc, \Bc'}L(\Bc', \lv) 
\end{equation*} 
with $a_{\Bc, \Bc'} \in \BQ(q)$, where $\Bc < \Bc'$ is the total order 
$\Bc <_p \Bc'$ on $\SC$ for $p = 0$.  
\end{prop}

\para{4.22.}
We now assume that $\Fg$ is simply laced.  
In this case, there exists the signed canonical basis 
$\wt\bB$ of $\BU_q^-$. Then we have the following. 

\begin{prop}  
Assume that $\Fg$ is simply laced.  Then 
\begin{enumerate}
\item 
$\SX_{\lv} = \{ L(\Bc, \lv) \mid \Bc \in \SC \}$ 
gives an $A$-basis of ${}_{\BA}\BU_q^-$.
\item 
For any $b \in \wt\bB$, there exists $\Bc \in \SC$ such that 
\begin{equation*}
b = \pm L(\Bc, \lv) + \sum_{\Bc < \Bc'}a_{\Bc,\Bc'}L(\Bc', \lv),  
\quad (a_{\Bc,\Bc'} \in q\BZ[q]).
\end{equation*} 
\end{enumerate}
\end{prop}

\begin{proof}
Since $\SX_{\lv}$ is a $\BQ(q)$-basis of $\BU_q^-$, $b$ can be written as 
a $\BQ(q)$-linear combination of elements in $\SX_{\lv}$.  By the upper 
triangularity of the bar involution for $\SX_{\lv}$ (Proposition 4.21), 
$b$ must be written as 
\begin{equation*}
b = a_{\Bc}L(\Bc, \lv) + \sum_{\Bc < \Bc'}a_{\Bc'}L(\Bc', \lv)
\end{equation*}
for some $\Bc \in \SC$, where $a_{\Bc}, a_{\Bc'} \in \BQ(q)$ 
with $\ol a_{\Bc} = a_{\Bc}$, $a_{\Bc} \ne 0$.
In particular, the transition matrix between $\wt\bB$ and $\SX_{\lv}$ 
is upper triangular.  Since $\wt\bB$ is an $\BA$-basis of ${}_{\BA}\BU_q^-$
and $L(\Bc, \lv) \in {}_{\BA}\BU_q^-$ for any $\Bc$ by Proposition 4.20, 
one can written as 

\begin{equation*}
L(\Bc,\lv) = \sum_{b}c_bb
\end{equation*}
with $c_b \in \BA$, 
where there exists $b_0 \in \wt\bB$ such that $c_{b_0} = a_{\Bc}\iv$, 
hence $\ol {c_{b_0}} = c_{b_0}$. 
Since $\wt\bB$ and $\SX_{\lv}$ satisfy the almost orthonormal property, 
we see that 
\begin{equation*}
\tag{4.23.1}
\sum_bc_b^2 \equiv 1 \mod q\BZ[q]
\end{equation*}
since $q\BZ[[q]] \cap \BA = q\BZ[q]$.
Let $t \ge 0$ be the smallest integer such that $q^tb \in \BZ[q]$ for any $b$.
Then $\sum_bc_b^2 \in q^{2t}\BZ[q]$, and by (4.23.1), we have $t = 0$.
It follows that $c_b \in \BZ[q]$, and $c_{b_0} \in \BZ$ since 
$\ol{c_{b_0}} = c_{b_0}$. This implies that $c_{b_0} = \pm 1$ and 
$c_b \in q\BZ[q]$ for $b \ne b_0$. 
In particular, the transition matrix between $\wt\bB$ and $\SX_{\lv}$ is upper
triangular, where the diagonal entries are $\pm 1$, and off-diagonal entries 
are contained in $q\BZ[q]$.  The proposition follows from this.  
\end{proof}

\para{4.24.}
Under the notation in Proposition 4.23, for each $\Bc \in \SC$, 
put $b(\Bc, \lv) = L(\Bc, \lv) + \sum_{\Bc < \Bc'}a_{\Bc,\Bc'}L(\Bc',\lv)$.
Then $\bB_{\lv} = \{ b(\Bc, \lv) \mid \Bc \in \SC \}$ gives the basis 
of $\BU_q^-$ 
such that $\bB_{\lv} \scup -\bB_{\lv} = \wt\bB$. 
In a similar way as in the proof of [BN, Thm. 3.13], 
it is shown that $\bB_{\lv}$ coincides with $\bB$ (see 1.23).   
\par\bigskip
\section{$\s$-action on root systems }

\para{5.1.}
Let $X = (I, (\ ,\ ))$ and $\ul X$ be as in 2.1.
The case where 
$\s$ preserves $I_0$, namely the case where $\ul X$ is twisted type, 
was discussed in [SZ].  So in this section, 
we consider the case where $\s(I_0) \ne I_0$.   
As is seen below, this covers all the cases where $\ul X$ is untwisted type.
Assume that $X$ is irreducible.  Then the pairs $(X, \ul X)$ are given 
as follows;

\par\medskip\noindent
(1) \ $X = D^{(1)}_{2n}$, \qquad $\ul X = B^{(1)}_n$, 
\par\noindent
(2) \ $X = D^{(1)}_n$, \qquad $\ul X = C^{(1)}_{n-2}$, 
\par\noindent
(3) \ $X = E^{(1)}_7$, \qquad $\ul X = F^{(1)}_4$, 
\par\noindent
(4) \ $X = E^{(1)}_6$, \qquad $\ul X = G^{(1)}_2$,  
\par\noindent
(5) \ $X = D^{(1)}_{2n+2}$, \quad\ $\ul X = A^{(2)}_{2n}$ \ ($n \ge 2$),
\par\noindent
(6) \ $X = D^{(1)}_4$, \qquad $\ul X = A^{(2)}_2$.  
\par\medskip
The order $\ve$ of $\s$ is 2 in  the cases (1) $\sim$ (3), $\ve = 3$ in the case (4), 
and $\ve = 4$ in the cases (5), (6).  
As in 2.2, we exclude the case where $\ve = 4$, hence we only consider 
the cases (1) $\sim$ (4). 

\para{5.2.}
Let $\eta_0$ be the $\s$-orbit in $I$ containing 0. 
Put $I_0' = I - \eta_0$ and let $\vD_0'$ be the subsystem of 
$\vD_0$ obtained from $I_0'$.  Then $I_0'$ and $\vD_0'$ are 
$\s$-stable. Let $\ul I_0'$ be the set of $\s$-orbits in $I_0'$.
Then we have $\ul I = \ul I_0' \scup \{ \eta_0 \}$, and $\ul I_0'$ 
corresponds to the Cartan datum $\ul X_0$ of finite type. 
For each $\b \in I$ let $O(\b) = \b + \s(\b) + \cdots \in Q^{\s}$
be the orbit sum of $\b$, and put 
$\ul\vD_0 = \{ O(\b) \mid \b \in \vD'_0\}$, which we consider 
as a subset of $\ul Q$ under the identification $Q^{\s} \simeq \ul Q$. 
Then it is known that $\ul\vD_0$ gives a root system of type $\ul X_0$
with simple system $\{ O_{\eta} \mid \eta \in \ul I_0'\}$, where
$O_{\eta} = O(\a_i)$ for $i \in \eta$. 
Let $\ul\vD \subset \ul Q$ be the root system for 
$\ul X$, and $\ul\vD^{\re,+}$ the set of positive real roots in $\ul\vD$.
Then $\ul\vD^{\re,+}$ can be identified with the subset of $Q^{\s}$ 
as follows; 
$\ul \vD^{\re,+} =  \ul\vD^{\re,+}_{>} \scup \ul\vD^{\re,+}_{<}$, where
\begin{equation*}
\ul\vD^{\re,+}_{>} = \{ m\d + \a \mid m \ge 0, \a \in \ul\vD^+_0 \}, \quad
\ul\vD^{\re,+}_{<} = \{ m\d - \a \mid m > 0, \a \in \ul\vD^+_0\}.
\end{equation*}    
Note that $\ul X$ is untwisted by 5.1, and so $r = 1$ in the notation of
(1.2.1). 

\par
Recall that $\a_0 = \d - \th$, where $\th$ is the highest root in $\vD_0^+$.
Let $\ul\th$ be the highest root in $\ul\vD^+_0$, and regard it as an 
element in $Q$. 
We consider $O(\a_0) \in Q$. 
The following relation can be verified by the case by case computation 
for the cases (1) $\sim$ (4). 
\begin{equation*}
\tag{5.2.1}
O(\a_0) = \d - \ul\th.
\end{equation*} 

(5.2.1) is equivalent to the formula 

\begin{equation*}
\tag{5.2.2}
O(\th) =  \begin{cases}
             \d + \ul\th &\quad\text{ if $\ve = 2$, } \\
            2\d + \ul\th &\quad\text{ if $\ve = 3$. }
          \end{cases} 
\end{equation*}

\para{5.3.}
Put $\vS^+_0 = \{ \a \in \vD^+_0 \mid O(\a) \in \BZ_{> 0}\d \}$.
We determine the set $\vS^+_0$ explicitly for the cases (1) $\sim$ (4). 
Take $\a = \sum_{i \in I_0}c_i\a_i \in \vD^+_0$.  Write $\th = \sum_{i \in I_0}a_i\a_i$.
First assume that $\ve = 2$ and that $\s : 0 \lra i_1$ for $\eta_0 = \{ 0, i_1\}$. 
We have $\s(\a_{i_1}) = \a_0 = \d - \th$, and   

\begin{align*}
\s(\a) &= \sum_{i \in I_0'}c_i\a_{\s(i)} + c_{i_1}(\d - \th) \\
       &= c_{i_1}\d - \sum_{i \in I_0'}(c_{i_1}a_i - c_{\s(i)})\a_i - c_{i_1}a_{i_1}\a_{i_1}.
\end{align*}
One can check, from the explicit information of the root system $D_n^{(1)}, E_7^{(1)}$, 
that $a_{i_1} = 1$ and $c_{i_1} = 1$ or 0. 
Hence in this case $\vS^+_0 = \{ \a \in \vD^+_0 \mid O(\a) = \d \}$, 
and the condition $\a \in \vS^+_0$ is given by 

\begin{equation*}
\tag{5.3.1}
c_i + c_{\s(i)} = a_i \quad (i \in I_0').
\end{equation*}

Next assume that $\ve = 3$, hence $X$ is of type $E_6^{(1)}$. 
Here $\s : 0 \to i_1 \to i_2 \to 0$ with $\eta_0 = \{ 0, i_1, i_2 \}$. 
In this case, $a_{i_1} = a_{i_2} = 1$.  Thus $c_{i_1}, c_{i_2} \in \{ 0, 1\}$. 
We have $\s(\a_{i_2}) = \a_0 = \d - \th$ with 

\begin{align*}
\tag{5.3.2}
\a + \s(\a) + \s^2(\a) =  &(c_{i_1}\a_{i_1} + c_{i_2}\a_{i_2})
               + (c_{i_1}\a_{i_2} + c_{i_2}\a_0) + (c_{i_1}\a_0 + c_{i_2}\a_{i_1})  \\
               &+ \sum_{i \in I_0'}(c_i + c_{\s(i)} + c_{\s^2(i)})\a_i
\end{align*} 
It follows from (5.3.2) that the condition $\a \in \vS^+_0$ is given by 

\begin{equation*}
\tag{5.3.3}
(c_{i_1} + c_{i_2})\a_i = c_i + c_{\s(i)} + c_{\s^2(i)} \quad (i \in I_0') 
\end{equation*}
and in that case, $\a + \s(\a) + \s^2(\a) = (c_{i_1} + c_{i_2})\d$. 
\par
We now consider each case separately.  
We fix a root system $\vD_0$ of type $D_n$ with 
vertex set $I_0 = \{ 1, 2, \dots, n\}$, where
$\vD^+_0 = \{ \ve_i \pm \ve_j \mid 1 \le i < j \le n \}$.
Here $\a_1 = \ve_1 - \ve_2, \dots, \a_{n-1} = \ve_{n-1} - \ve_n$,
$\a_n = \ve_{n-1} + \ve_n$ gives the set of simple roots. We have 
$\th = \a_1 + 2(\a_2 + \cdots + \a_{n-2}) + \a_{n-1} + \a_n$
\par\medskip
Case (1) : $X$ is of type $D_{2n}^{(1)}$ and $\ul X$ is of type $B_n^{(1)}$.
$\s : i \lra 2n-i$ for $i = 0, 1, \dots, 2n$.  Thus $\eta_0 = \{ 0, 2n \}$, 
$I_0' = \{ 1,2, \dots, 2n-1\}$ and $\vD_0'$ is of type $A_{2n-1}$. 
Moreover, $\ul I_0' = \{ \ul 1, \dots, \ul n\}$ and $\ul \vD_0$ is of type $B_n$
with $\a_{\ul n}$ short root, and  
$\ul\th = \a_{\ul 1} + 2\a_{\ul 2} + \cdots + 2\a_{\ul n}$.  
The condition (5.3.1) implies that $c_i + c_{\s(i)} = 2$ for $i = 2, \dots, n-1$,
and $c_n = 1, c_{2n} = 1$.  Thus we have
\begin{equation*}
\tag{5.3.4}
\vS^+_0 = 
 \{ \ve_1 + \ve_{2n}, \ve_2 + \ve_{2n-1}, \dots, \ve_n + \ve_{n+1} \}
\end{equation*}
In particular $\vS^+_0$ 
gives a positive subsystem of $\vD^+_0$ of type $nA_1$. 
\par\medskip
Case (2) : $X$ is of type $D_n^{(1)}$ and $\ul X$ is of type $C^{(1)}_{n-2}$. 
$\s : 0 \lra 1, n-1 \lra n$, and $\s(i) = i$ otherwise. Thus 
$\eta_0 = \{ 0, 1\}$, $I_0' = \{ 2,3, \dots, n\}$ and $\vD_0'$ is of type $D_{n-1}$.
Moreover, $\ul I_0' = \{ \ul 2, \dots, \ul {n-2}, \ul{n-1}\}$ and 
$\ul\vD_0$ is of type $C_{n-2}$ with $\a_{\ul{n-1}}$ long root. We have 
$\ul\th = 2\a_{\ul 2} + \cdots + 2\a_{\ul{n-2}} + \a_{\ul{n-1}}$.
(5.3.1) implies that $c_i = c_{\s(i)} = 1$ for $i = 2, \dots, n-2$
and $(c_{n-1}, c_n) = (1,0)$ or $(0,1)$.  Thus we have
\begin{equation*}
\tag{5.3.5}
\vS^+_0 = \{ \ve_1 \pm \ve_n \}.
\end{equation*} 
In particular $\vS^+_0$ 
gives a positive subsystem of type $2A_1$. 
\par\medskip
Case(3) : $X$ is of type $E^{(1)}_7$, and 
$\ul X$ is of type $F^{(1)}_4$. 
$I_0$ is of type $E_7$, and we fix $I_0 = \{ 1,2, \dots, 7\}$ as in the 
table of Kac [Ka, p.53].  Here $\s : I \to I$ is given by
$i \lra 6-i$ for $i = 0, \dots, 6$, and $\s(7) = 7$.  
Then $\eta_0 = \{ 0, 6\}$, and $I_0' = \{ 1, 2, \dots, 5, 7\}$. 
$\vD'_0$ is of type $E_6$. 
$\th = 2\a_1 + 3\a_2 + 4\a_3 + 3\a_4 + 2\a_5 + \a_6 + 2\a_7$. 
Moreover, $\ul I_0' = \{ \ul 1, \ul 2, \ul 3, \ul 7\}$, and $\ul\vD_0$ is of type
$F_4$ with $\a_{\ul 1}$ long root. 
Here $\ul \th = 2\a_{\ul 1} + 3\a_{\ul 2} + 4\a_{\ul 3} + 2\a_{\ul 7}$.   
In this case, $i_1 = 6$ and (5.3.1) implies that 
\begin{equation*}
\begin{cases}
c_3 &= 2, \\
c_6 &= 1, \\
c_7 &= 1, \\
c_1 + c_5 &= 2, \\
c_2 + c_4 &= 3.
\end{cases}
\end{equation*}
 
Thus by using the explicit description of positive roots in $E_7$ 
in [Bo], we have
\begin{equation*}
\tag{5.3.6}
\vS^+_0 = \biggl\{ \b_1 = \begin{matrix} 122111  \\
                                  1\phantom{*} 
                       \end{matrix}, \quad
                       \b_2 = \begin{matrix} 112211  \\
                                   1\phantom{*}
                       \end{matrix}, \quad 
                       \b_3 = \begin{matrix}  012221  \\
                                   1\phantom{*}
                       \end{matrix}  \biggr \}.              
\end{equation*}
Under the notation in [Bo], one can write as 
\begin{align*}
\b_1 &= \frac{1}{2}(\ve_8 - \ve_7 -\ve_1 + \ve_2 + \ve_3 - \ve_4 - \ve_5 + \ve_6), \\
\b_2 &= \frac{1}{2}(\ve_8 - \ve_7 + \ve_1 - \ve_2 - \ve_3 + \ve_4 - \ve_5  + \ve_6), \\
\b_3 &= \ve_5 + \ve_6.
\end{align*}
In particular, $(\b_i,\b_j) = 0$ for $i \ne j$. 
Thus $\vS^+_0$ is a positive subsystem of $\vD^+_0$ of type $3A_1$. 

\par\medskip
Case(4) : 
$X$ is of type $E_6^{(1)}$ and $\ul X$ is of type $G_2^{(1)}$.
$I_0$ is of type $E_6$, and we fix $I_0 = \{ 1,2, \dots, 6\}$ as in the table
of [K, p.53]. Here $\s : I \to I$ is given by 
$0 \to 5 \to 1 \to 0, 6 \to 4 \to 2 \to 6$, and $\s(3) = 3$.
Then $\eta_0 = \{ 0, 1, 5 \}$, and $I_0' = \{ 2,3,4,6 \}$. $\vD'_0$ is of type 
$D_4$. $\th = \a_1 + 2\a_2 + 3\a_3 + 2\a_4 + \a_5 + 2\a_6$. 
Moreover, $\ul I_0' = \{ \ul 2, \ul 3\}$, and $\ul\vD_0$ is of type $G_2$
with $\a_{\ul 2}$ long root.  Here $\ul\th = 2\a_{\ul 2} + 3\a_{\ul 3}$.

Thus the condition (5.3.3) can be written as
\begin{align*}
\tag{5.3.7}
c_2 + c_4 + c_6 &= 2(c_1 + c_5), \\
3c_3 &= 3(c_1 + c_5).
\end{align*}
In particular, $c_3 = c_1 + c_5$.
We must have $c_1 + c_5 > 0$.  Thus $(c_1, c_5) = (1,0), (0,1), (1,1)$.
Let $A_1, A_2, A_3$ be the set of $\a \in \vD^+_0$ satisfying (5.3.7) for 
$(c_1,c_5) = (1,0), (0,1)$ or $(1,1)$, respectively.  
By using the explicit description of the root system of type $E_6$ in [Bo], 
We have
\begin{align*}
A_1 &=  \biggl\{ 
         \b_1 = \begin{matrix}
    11110 \\
      0
          \end{matrix}, \quad
          \b_2 = \begin{matrix}
    11100  \\
      1     
          \end{matrix}\biggr\}, \\  
A_2 &= \biggl\{ 
          \b_3 = \begin{matrix}
             01111  \\
               0
                 \end{matrix}, \quad
          \b_4 = \begin{matrix}
             00111   \\
               1
           \end{matrix}\biggr\},  \\
A_3 &= \biggl\{ \b_5 = \begin{matrix}
        12211  \\
          1
\end{matrix}, \quad 
\b_6 = \begin{matrix}
        11221  \\
          1
\end{matrix} \biggr\}. 
\end{align*}
Summing up the above, we have 
\begin{align*}
\tag{5.3.8}
\vS^+_0 = \{ \b_1, \b_2, \b_3, \b_4, \b_5, \b_6 \}, 
\end{align*}
where 
\begin{align*}
\tag{5.3.9}
\b_1 &= \frac{1}{2}(-\ve_1 -\ve_2 -\ve_3 + \ve_4 - \ve_5 - \ve_6 - \ve_7 +\ve_8), \\
\b_2 &= \frac{1}{2}(\ve_1 + \ve_2 + \ve_3 -\ve_4 -\ve_5 -\ve_6 -\ve_7 +\ve_8), \\ 
\b_3 &= -\ve_1 + \ve_5, \\
\b_4 &= \ve_1 + \ve_5, \\
\b_5 &= \frac{1}{2}(-\ve_1 + \ve_2 + \ve_3 - \ve_4 + \ve_5 - \ve_6 -\ve_7 + \ve_8), \\
\b_6 &= \frac{1}{2}(\ve_1 -\ve_2 -\ve_3 +\ve_4 +\ve_5 -\ve_6 -\ve_7 + \ve_8).
\end{align*}

Then if we put $X_1 = \{ \b_1, \b_4, \b_6\}, X_2 = \{ \b_2, \b_3, \b_5\}$, 
we have $(\b, \b') = 0$ for $\b \in X_1, \b' \in X_2$, and 
$\b_1 + \b_4 = \b_6, \b_2 + \b_3 = \b_5$.  Thus $\vS^+_0 = X_1\scup X_2$ gives 
a positive subsystem  of type $2A_2$ of $\vD_0^+$.

\para{5.4.}
Let $\ul\vD_{0,l}$ (resp. $\ul\vD_{0,s}$) be the set of long roots 
(resp. short roots) in $\ul\vD_0$. 
We denote by 
$\ul\vD^+_{0,l}, \ul\vD^+_{0,s}$ the corresponding positive roots.  
\par
Put 
\begin{align*}
\tag{5.4.1}
\vS^+_s &= \{ \a \in \vD^+_0 \mid O(\a) \in \ul\vD_{0,s}^+ \}, \\
\vS^+_{l} &= \{ \a \in \vD^+_0 \mid O(\a) \in \ul\vD_{0,l}^+ \}, \\ 
\vS^+_{l'} &= \{ \a \in \vD^+_0 \mid O(\a) \in \d + \ul\vD^+_{0,l} \}, \\
\vS^+_{l''} &= \{ \a \in \vD^+_0 \mid O(\a) \in 2\d + \ul\vD^+_{0,l} \}, \\ 
\vS^+_0 &= \{ \a \in \vD^+_0 \mid O(\a) \in \BZ_{>0}\d \}.
\end{align*} 
($\vS^+_0$ is already defined in 5.3.) 
We also put
\begin{align*}
\tag{5.4.2}
\vS_{l'} &= \{ \a \in \vD^+_0 \mid O(\a) \in \d + \ul\vD_{0,l} \}, \\
\vS_{l''} &= \{ \a \in \vD^+_0 \mid O(\a) \in 2\d + \ul\vD_{0,l}\}.
\end{align*}
  
Note that $\vS^+_{l''}$ and $\vS_{l''}$ are empty sets if $\ve = 2$. 
Also note that 
$\vS^+_s$ and $\vS^+_l$ are $\s$-stable, and $\s$ acts trivially on $\vS^+_s$.
\par
Let $\ul W_0$ be the Weyl group associated to $\ul \vD_0$.  Then 
$\ul W_0$ is a subgroup of $\ul W$, and under the identification 
$\ul W \simeq W^{\s}$, $\ul W_0$ is regarded as a subgroup of $W$. 
We define a subset $E$ of $\vD_0$ as follows; in the case where $\ve = 2$, 
$E$ is the $\ul W_0$-orbit of $\s(\a_0)$ in $\vD_0$ 
(note that $\s(\a_0) \in \vD_0$), while in the case where $\ve = 3$, 
$E$ is the union of $\ul W_0$-orbit of $\s(\a_0)$ and of $\th$.
We put $E^+ = E \cap \vD_0^+$.
\par
The following lemma holds. 
  
\begin{lem}  
Under the notation above, 
\begin{enumerate}
\item 
$\vS_{l'} \scup \vS_{l''} = E^+$.
\item 
$\vD^+_0 = \vS^+_s \scup \vS^+_{l} \scup \vS_{l'} 
         \scup \vS_{l''}  \scup \vS^+_0$.  
\end{enumerate}
\end{lem}

\begin{proof}
We know  $O(\a_0) = \d - \ul\th$ by (5.2.1) , and $O(\th) = 2\d + \ul\th$ 
(for the case $\ve = 3$) by (5.2.2). Since any long root in 
$\ul\vD_0$ is conjugate to $\ul\th$ under $\ul W_0$, we see that, 
for any $\a \in E^+$,  
$O(\a)$ can be written as $\d - \ul\b$ or $2\d + \ul\b$, 
where $\ul\b \in \ul\vD_{0,l}$.  
Hence $\a \in \vS_{l'} \scup \vS_{l''}$. 
On the other hand, in the case $\ve = 2$, if $\a \in \vD^+_0$ satisfies the condition that 
$O(\a) = \d - \ul\b$ with $\ul\b \in \ul\vD_{0,l}$, 
then the coefficient of $\s(\a_0)$ in $\a$ is non-zero, 
while in the case $\ve = 3$, if $O(\a) = \d - \ul\b$ or 
$O(\a) = 2\d  + \ul\b$ for some $\ul\b \in \ul\vD_0$, then the coefficient of 
$\s(\a_0)$ or $\s^2(\a_0)$ is non-zero. 
It follows that $\a \in \vD^+_0 - (\vD'_0)^+$. 
Hence we have 
\begin{equation*}
\tag{5.5.1}
E^+ \subseteq \vS_{l'} \cup \vS_{l''} 
            \subseteq \vD_0^+ - (\vD'_0)^+ - \vS^+_0. 
\end{equation*}   
Assume that $\ve = 2$. By (5.2.1),  
for each $\ul\b \in \ul\vD_{0,l}$ one can find 
a unique $\a \in E^+$ such that $O(\a) = \d - \ul\b$.  
Hence $|E^+| =  |\ul\vD_{0,l}|$. 
On the other hand, assume that $\ve = 3$.  By (5.2.1) and (5.2.2), 
for each $\ul\b$, one can find 
two $\a \in E^+$ such that $O(\a) = \d - \ul\b$, and a unique $\a \in E^+$ such that
$O(\a) = 2\d + \ul\b$. Hence $|E^+| = 3|\ul\vD_{0,l}|$. 
It follows that 
\begin{equation*}
\tag{5.5.2}
c|\ul\vD_{0,l}| =  |E^+| \le |\vD^+_0| - |(\vD'_0)^+| - |\vS^+_0|,
\end{equation*} 
where $c = 1$ (resp. $c = 3$) if $\ve = 2$ (resp. $\ve = 3$). 
We note that
\par\medskip\noindent
(5.5.3) \  The equality holds for (5.5.2).
\par\medskip
This assertion can be verified by the case by case computation as follows. 
\par\medskip
Case (1) : $\vD_0$ : type $D_{2n}$, $\vD'_0$ : type $A_{2n-1}$, 
and $\ul\vD_0$ : type $B_n$.  Hence,  
$|\vD^+_0| = 2n(2n-1)$, $|(\vD'_0)^+| = n(2n-1)$ and $|\ul\vD_{0,l}| = 2n(n-1)$.
By (5.3.4), $|\vS^+_0| = n$. Thus the equality holds. 
\par\medskip
Case (2) : $\vD_0$ : type $D_n$, $\vD'_0$ : type $D_{n-1}$, 
$\ul\vD_0$ : type $C_{n-2}$.  Hence
$|\vD^+_0| = n(n-1)$, $|(\vD'_0)^+| = (n-1)(n-2)$, and 
$|\ul\vD_{0,l}| = 2(n-2)$.
By (5.3.5), $|\vS^+_0| = 2$. Thus the equality holds. 
\par\medskip
Case (3) : $\vD_0$ : type $E_7$, $\vD'_0$ : type $E_6$, 
$\ul\vD_0$ : type $F_4$.  Hence 
$|\vD^+_0| = 63$, $(\vD'_0)^+| = 36$ and $|\ul\vD_{0,l}| = 24$.
Moreover, $|\vS^+_0| = 3$ by (5.3.6).
Thus the equality holds.    
\par\medskip
Case (4) : $\vD_0$ : type $E_6$, $\vD'_0$ : type $D_4$, 
$\ul\vD_0$ : type $G_2$.  Hence 
$|\vD^+_0| = 36$, $|(\vD'_0)^+| = 12$, $|\ul\vD_{0,l}| = 6$.  
Moreover, $|\vS^+_0| = 6$ by (5.3.8).  Thus the equality  holds.    
\par\medskip
By (5.5.1) and (5.5.3), we have

\begin{equation*}
\tag{5.5.4}
E^+ = \vS_{l'} \cup \vS_{l''} = \vD_0^+ - (\vD'_0)^+ - \vD^+_0.
\end{equation*}
(i) follows from this. 
Since $(\vD'_0)^+ = \vS^+_{l} \scup \vS^+_s$, (ii) also follows.
The lemma is proved. 
\end{proof}

\para{5.6.}
For each $m \ge 0$, and for $x \in \{ s, l, l',l''\}$, 
we define subsets $\vS_x(m), \vS'_x(m)$ of $\vD^+$ as 
\begin{align*}
\tag{5.6.1}
\vS_x(m) &= \bigcup_i\s^i \{ m\d + \a \mid \a \in \vS^+_x \}, \\  
\vS'_x(m) &= \bigcup_i\s^i\{ m\d - \a \mid \a \in \vS^+_x \}.
\end{align*}
(Here and below, we assume $m \ge 1$ for $\vS'_x(m)$ or $\ul\vS'_x(m)$.)
We denote by 
$\wh\vS_x(m)$ the set of $\s$-orbits in $\vS_x(m)$, and 
define $\wh\vS'_x(m)$ similarly. 
For $x = 0$ and $m \ge 0$, we put
\begin{equation*}
\tag{5.6.2}
\vS_0(m) = \{ m\d + \a \mid \a \in \vS^+_0 \}, \qquad 
\vS'_0(m) = \{ m \d - \a \mid \a \in \vS^+_0 \}. 
\end{equation*}
In the case where $\ve = 2$, we have
$\s(\vS_0(m)) = \vS'_0(m + 1)$. 
While in the case where $\ve = 3$, the following holds.
Recall that $\vS^+_0 = A_1 \scup A_2 \scup A_3$ in the notation of 
(5.3.8).
Assume that $\b = m\d + \a \in \vS_0(m)$.   
If $\a \in A_1$, then $\s(\b) \in \vS_0(m), \s^2(\b) \in \vS'_0(m +1)$. 
If $\a \in A_2$, then $\s(\b) \in \vS'_0(m +1), \s^2(\b) \in \vS_0(m )$. 
If $\a \in A_3$, then $\s(\b) \in \vS'_0(m +1), \s^2(\b) \in \vS'_0(m + 1)$. 
\par 
Also for each $m \ge 0$, we define subsets of $\ul\vD^{\re,+}_{>}$ by 

\begin{equation*}
\tag{5.6.3}
\begin{aligned}
\ul\vS_{s}(m) &= \{ m\d + \a \mid \a \in \ul\vD^+_{0,s} \}, 
       &\quad \ul\vS'_s(m) &= \{ m\d - \a \mid \a \in \ul\vD^+_{0,s} \}  \\
\ul\vS_{l}(m) &= \{ m\d + \a \mid \a \in \ul\vD^+_{0,l} \},
       &\quad \ul\vS'_l(m) &= \{ m\d - \a \mid \a \in \ul\vD^+_{0,l} \}.
\end{aligned}
\end{equation*}

We have the following result.

\begin{prop}  
Under the identification $Q^{\s} \simeq \ul Q$, the following holds. 
\begin{enumerate}
\item
For each $\b \in \vD^+$, $O(\b) \in \ul\vD^+$. 
We define a map $f : \vD^+ \to \ul\vD^+$ by $\b \mapsto O(\b)$. 
\item
Assume that $\ve = 2$.  For each $m \ge 0$ $($resp. $m \ge 1$ $)$ for 
$\wh\vS_x(m)$ 
$($resp. $\wh\vS'_x(m))$, $f$ induces a bijection 
\begin{equation*}
\begin{aligned}
\wh\vS_s(m) &\isom  \ul\vS_s(m),  &\quad 
\wh\vS'_s(m) &\isom  \ul\vS'_s(m),  \\
\wh\vS_l(m) &\isom  \ul\vS_l(2m),  &\quad 
\wh\vS'_l(m) &\isom  \ul\vS'_l(2m),   \\ 
\wh\vS_{l'}(m) &\isom  \ul\vS_l(2m + 1), &\quad  
\wh\vS_{l'}(m) &\isom  \ul\vS_l(2m - 1).
\end{aligned}
\end{equation*}
\item
Assume that $\ve = 3$.  For each $m \ge 0$ $($resp. $m \ge 1$ $)$ for 
$\wh\vS_x(m)$ 
$($resp. $\wh\vS'_x(m))$, $f$ induces a bijection 
\begin{equation*}
\begin{aligned}
\wh\vS_s(m) &\isom  \ul\vS_s(m), &\quad 
\wh\vS'_s(m) &\isom  \ul\vS'_s(m), \\
\wh\vS_l(m) &\isom  \vS_l(3m), &\quad 
\wh\vS'_l(m) &\isom  \vS'_l(3m),  \\
\wh\vS_{l'}(m) &\isom  \ul\vS_{l}(3m +1), &\quad 
\wh\vS'_{l'}(m) &\isom  \ul\vS'_{l}(3m -1), \\
\wh\vS_{l''}(m) &\isom  \ul\vS_{l}(3m + 2) &\quad 
\wh\vS'_{l''}(m) &\isom  \ul\vS'_{l}(3m - 2).
\end{aligned}
\end{equation*}
\end{enumerate}
\end{prop}

\begin{proof}
For the proof of (i), it is enough to check this for 
$\b \in \vD^+_0$. But this is clear from Lemma 5.5 (ii). 
(ii) and (iii) also follow from Lemma 5.5.
In the following, we only consider the case $\wh\vS_x(m)$.  The case 
$\wh\vS'_x(m)$ is similar. 
The assertion is clear for $\wh\vS_s(m) = \vS_s(m)$. 
Now assume that $\ve = 2$. If $\ul\b = 2m\d + \ul\a \in \ul\vS_l(2m)$, 
then $f\iv(\ul\b) = \{ m\d + \a, m\d + \a' \}$, where 
$\a \in \vD^+_0$ is such that $\ul \a = O(\a) = \a + \a' \in \ul\vD^+_{0,l}$.  
Thus $f\iv(\b)$ gives an orbit in $\vS_{l}(m)$. 
If $\ul\b = (2m+1)\d + \ul\a \in \ul\vS_{l}(2m+1)$, 
then $f\iv(\ul\b) = \{ m\d + \a, (m+1)\d + \a'\}$, where 
$\a \in \vD^+_0$ is such that $\ul\a = O(\a) = \a + \a' \in \d + \ul\vD^+_{0,l}$.
Thus $f\iv(\ul \b)$ gives an orbit in $\vS_{l'}(m)$.  
Next assume that $\ve = 3$. 
If $\ul\b = 3m\d + \ul\a \in \ul\vS_l(3m)$, 
then $f\iv(\ul\b) = \{ m\d + \a, m\d + \a', m\d + \a''\}$, where 
$\ul \a = O(\a) = \a + \a' + \a'' \in \ul\vS_{l}(3m)$.  
If $\ul\b = (3m+1)\d + \ul\a \in \ul\vS_l(3m + 1)$, 
then $f\iv(\ul\b) = \{ m\d + \a, m\d + \a', (m+1)\d + \a''\}$, 
where $\ul\a = O(\a) = \a + \a' + \a'' \in \ul\vS_{l}(3m+1)$. 
If $\ul\b = (3m+2)\d + \a \in \vS_l(3m+2)$,  
then $f\iv(\ul\b) = \{ m\d + \a, (m +1)\d + \a', (m +1)\d + \a''\}$, 
where $\ul\a = O(\a) = \a + \a' + \a'' \in \ul\vS_l(3m+2)$. 
This gives the required bijection. 
\end{proof}

\par\bigskip
\section{ The correspondence for PBW-bases}

\para{6.1.}
Let $\SX_{\Bh}$ be the PBW-basis of $\BU_q^-$ associated to $\Bh$, and 
$\ul\SX_{\ul\Bh}$ the PBW -basis of $\ul\BU_q^-$ associated to $\ul\Bh$ 
as given in Section 1. 
In the case where $\s$ preserves $I_0$ it was shown in [SZ] 
that 
$\s$ acts on $\SX_{\Bh}$ as a permutation, and 
the subset of $\s$-fixed elements in  $\SX_{\Bh}$ is in bijection 
with $\ul\SX_{\ul\Bh}$ through the surjective homomorphism 
$\f : {}_{\BA'}\BU_q^{-,\s} \to {}_{\BA'}\ul\BU_q^-$.  
We want to extend this result to the case where 
$\s$ does not preserve $I_0$. 
However, in this case, the PBW-basis $\SX_{\Bh}$ does not
fit well with respect to the $\s$-action.  In this section, 
we shall construct a new type of PBW-basis $\SX_{\lv}$ of $\BU_q^-$ associated to 
a certain convex order $\lv$ on $\vD^{\min}_+$ by applying 
the discussion in Section 4, and consider 
the relationship between its $\s$-fixed part and a PBW-basis $\ul\SX_{\ul\Bh}$  
of $\ul\BU_q^-$. 
\para{6.2.}
We consider the doubly infinite sequence 
$\ul\Bh = (\dots, \eta_{-1}, \eta_0, \eta_1, \dots)$
for $\ul W$ associated to $\ul\xi \in \ul P_{\cl}$ as defined in 1.5, 
applied for $\ul\BU_q^-$.  
We define a sequence 
$\Bh' = (\dots, i_{-1}, i_0, i_1, \dots)$ by replacing $\eta = \eta_i$
by the corresponding subset $\{ j_1, \dots, j_{|\eta|}\}$ of $I$, namely, 
\begin{equation*}
\tag{6.2.1}
\Bh'     = (\dots, \underbrace{i_{-|\eta_{-1}|}, \dots, i_{-1}}_{\eta_{-1}}
                 \underbrace{i_0, \dots, i_{|\eta_0|}}_{\eta_{0}}
                \underbrace{i_{|\eta_0|+1}, \dots, i_{|\eta_0| + |\eta_1|}}_{\eta_1}, 
               \dots ).
\end{equation*}
For any $k < \ell$, $\ul w = s_{\eta_k} \cdots s_{\eta_{\ell}}$ is a reduced expression
for $\ul w \in \ul W$.  We define $w \in W$ by 
\begin{equation*}
\tag{6.2.2}
w = \biggl(\prod_{i \in \eta_k}s_i\biggr)\cdots 
                  \biggl(\prod_{i \in \eta_{\ell}}s_i\biggr).
\end{equation*}
In view of Proposition 5.7, we see that 
\begin{equation*}
\vD^+ \cap w\iv(-\vD^+) = f\iv(\ul\vD^+ \cap \ul w\iv (-\ul \vD^+)). 
\end{equation*}   
It follows that  
(6.2.2) is a reduced expression of $w \in W$,  
and $\Bh'$ gives an infinite reduced word for $W$ as in 4.4. 
Note that $\ul\b_k \in \ul\vD^+$ is defined as in (1.5.1) by using $\ul\Bh$. 
We define  $\b_k \in \vD^+$ for $k \in \BZ$ as in 4.4 by using $\Bh'$.  
By (1.5.2), we know  that 
\begin{equation*}
\tag{6.2.3}
\{ \ul\b_k \mid k \le 0 \} = \ul\vD^+_{>}, \quad 
\{ \ul\b_k \mid k > 0\}   = \ul\vD^+_{<}.
\end{equation*}

Recall $\vS_x(m), \vS'_x(m)$ and $\vS_0(m), \vS'_0(m)$ in (5.6.1) 
and (5.6.2). 
We define $\vD^{(0)}_{>}, \vD^{(0)}_{<}$ by 
\begin{equation*}
\tag{6.2.4}
\vD^{(0)}_{>} = \bigsqcup_{m \ge 0}\bigsqcup_{x \in \{ s, l, l', l''\}}\vS_x(m), 
\quad
\vD^{(0)}_{<} = \bigsqcup_{m \ge 1}\bigsqcup_{x \in \{ s,l,l',l''\}}\vS'_x(m).
\end{equation*}
Also put
\begin{equation*}
\tag{6.2.5}
\vD^{(1)}_{>} = \bigsqcup_{m \ge 0}\vS_0(m), \quad
\vD^{(1)}_{<} = \bigsqcup_{m \ge 1}\vS'_0(m). 
\end{equation*}
Then $\vD^{(0)}_{>}$ and $\vD^{(0)}_{<}$ are $\s$-stable, and in the case where
$\ve  = 2$, $\s$ permutes $\vD^{(1)}_{>}$ and $\vD^{(1)}_{<}$ 
(in the case where $\ve = 3$, $\vD^{(1)}_{>}\scup \vD^{(1)}_{<}$ is $\s$-stable).  
\par
By Proposition 5.7, we have
\begin{align*}
\tag{6.2.6}
\{ \b_k \mid k \le 0\} = \vD^{(0)}_{>}, \quad
\{ \b_k \mid k > 0 \} = \vD^{(0)}_{<}. 
\end{align*}
Moreover, we have 
$\vD^+ = \vD^{(0)}_{>} \scup \vD^{(1)}_{>} \scup \vD^{(1)}_{<} \scup \vD^{(0)}_{<}$. 

\para{6.3.}
We define a total order $\lv$ on $\vD^{(0)}_{>}$ by 
$\b_0 \lv \b_{-1} \lv \b_{-2} \cdots$, and a total order $\lv$ on 
$\vD^{(0)}_{<}$ by 
$\cdots \lv \b_2 \lv \b_1$.  Then by Lemma 4.3, $\lv$ gives a convex
order on $\vD^{(0)}_{>}$ and a reverse convex order on $\vD^{(0)}_{<}$. 
This implies that
\par\medskip\noindent
(6.3.1) \  $\vD^{(0)}_{>} \lv \d \lv \vD^{(0)}_{<}$ satisfies the condition 
(4.2.1).
\par\medskip
We shall define a total order on $\vD^{(1)}_{>}$ and $\vD^{(1)}_{<}$.  
First assume that $\ve = 2$.  We write 
$\vS^+_0$ as $\{ \g_1, \dots, \g_t \}$ in any order. 
We define a total order on $\vD^{(1)}_{>}$ by 
the condition that $m\d + \g_i \lv m\d + \g_j$ for $i < j$, and 
$m\d + \g_i \lv m'\d + \g_j$ for $m < m'$ and for any $i,j$. 
Similarly, we define a total order on $\vD^{(1)}_{<}$ by the condition that 
$m\d - \g_i \lv m\d - \g_j$ for $i > j$, and $m\d - \g_i \lv m'\d - \g_j$ 
for $m > m'$ and for any $i,j$.
 
\par
Next assume that $\ve = 3$.
Then $\vS^+_0$ is given by (5.3.8). 
$\vS^+_0 = X_1 \scup X_2$ in the notation there.  We write them as  
$X_1 = \{ \g_1, \g_2, \g_3\}, X_2 = \{ \g_1', \g_2', \g_3'\}$, where 
$\g_3 = \g_1 + \g_2, \g'_3 = \g'_1 + \g'_2$, and 
they give positive system of type $A_2$.  
Let $\vD^{(1)}$ be the subsystem of $\vD$ spanned by $\g_1, \g_2, \g_0 = \d - \th_1$, 
and $\g_1', \g_2', \g_0' = \d - \th_1'$ 
where $\th_1 = \g_3, \th_1' = \g'_3$ are the highest roots in $X_1, X_2$. 
Thus $\vD^{(1)}$ is the affine root system of type $A_2^{(1)} \times A_2^{(1)}$.  
We denote by 
$W_1$ the Weyl group of $\vD^{(1)}$.  We consider the infinite reduced 
word $\Bh^{(1)} = (\dots, i_{-1}, i_0, i_1, \dots)$ of $W_1$ defined by [BN] as in 1.5, 
and define $\b'_k \in (\vD^{(1)})^{\re,+}$ for $k \in \BZ$ as in (1.5.1).  Then  
one can check that 
\begin{equation*}
\tag{6.3.2}
\{ \b'_k \mid k \le 0\} = \vD^{(1)}_{>}, \quad 
\{ \b'_k \mid k > 0 \} = \vD^{(1)}_{<}.
\end{equation*} 
We define a total order on $\vD^{(1)}_{>}$ by 
$\b_0' \lv \b'_{-1} \lv \b'_{-2} \lv \cdots$, and a total order on 
$\vD^{(1)}_{<}$ by $\cdots \lv \b'_2 \lv \b'_1$.  
\par
For the case either $\ve = 2$ or $\ve = 3$, we have 
\par\medskip\noindent
(6.3.3) \ $\vD^{(1)}_{>} \lv \d \lv \vD^{(1)}_{<}$ satisfies 
the condition (4.2.1).
\par\medskip\noindent
In fact, in the case $\ve = 3$, (6.3.3) follows from Lemma 4.3.
The definition of the orders on 
$\vD^{(1)}_{>}, \vD^{(1)}_{<}$ as in the case where $\ve = 3$ is also
available for the case $\ve = 2$. Thus (6.3.3) holds also for $\ve = 2$. 
(In the case where $\ve = 2$, this can be checked directly as follows.
It follows from the result in 5.3 that $(\g_i, \g_j) = 0$ for $i \ne j$.  
Since $\vD_0$ is simply laced, $\g_i + \g_j \notin \vD_0$ for any $i \ne j$.  
This implies that $\b +  \b'$ is not contained in $\vD^{\re}$ for any 
$\b, \b' \in \vD^{(1)}_{>} \scup \vD^{(1)}_{<}$.  Thus (6.3.3) holds.) 
\par
We now define a total order $\lv$ on $\vD^{\min}_+$ by extending the total order
on $\vD^{(i)}_{>}, \vD^{(i)}_{<}$ for $i =1,2$, under the condition that 
\begin{equation*}
\tag{6.3.4}
\vD^{(0)}_{>} \lv \vD^{(1)}_{>} \lv \d \lv \vD^{(1)}_{<} \lv \vD^{(0)}_{<}.
\end{equation*}

\begin{lem}  
$\lv$ gives a convex order of the standard type on $\vD^{\min}_+$. 
\end{lem}

\begin{proof}
We show that $\lv$ gives a convex order on $\vS^{\min}_+$.  Then it is clear
from (6.3.4) that it is of the standard type. 
Take  $\b \in \vD^{(0)}_{>}$ (resp. $\b \in \vD^{(0)}_{<}$), 
$\b' \in \vD^{(1)}_{>} \scup \vD^{(1)}_{<}$, and assume that 
$\g = \b + \b' \in \vD^+$. We show that $\g \in \vD^{(0)}_{>}$
and $\b \lv \g$ (resp. $\g \in \vD^{(0)}_{<}$ and $\g \lv \b$). 
In fact take $\b \in \vD^{(0)}_{>}$. If $\ve = 2$, we have 
$\g + \s(\g) = (\b + \s(\b)) + (\b' + \s(\b')) \notin \BZ_{>0}\d$ 
since $(\b + \s(\b)) \notin \BZ_{>0}\d$, $(\b' + \s(\b')) \in \BZ_{>0}\d$. 
Thus $\g \in \vD^{(0)} \scup \vD^{(0)}_{<}$.  
This is shown also for the case where $\ve = 3$, by using a similar argument
as above. 
If $\b' \in  \vD^{(1)}_{>}$, 
we must have $\g \in \vD^{(0)}_{>}$ and $\b \lv \g$.  So, consider the case where  
$\b' \in \vD^{(1)}_{<}$.  Suppose that $\g \in \vD^{(0)}_{<}$. 
We can write as $\b = m\d + \a, \b' = m'\d - \a'$ and $\b + \b' = m''\d -\a''$,
where $\a, \a'' \in \vS^+_{x}$ with $x \ne 0$, and $\a' \in \vS^+_0$. 
Then we have $\a + \a'' = \a'$.  This contradicts to $\a' \in \vS^+_0$. 
Hence $\g \in \vD^{(0)}_{>}$ and $\b \lv \g$.  The case where $\b \in \vD^{(0)}_{<}$
is proved similarly.    
\par
The lemma now follows from  (6.3.1) and (6.3.3). 
\end{proof}

\remark{6.5.}
The convex order on $\vD^{\min}_+$ defined above 
is the two-rows convex order in the terminology of [I2]. 

\para{6.6.}
From now on we fix the convex order $\lv$ on $\vD^{\min}_+$ given in 
Lemma 6.4.  
We consider the PBW-basis $\SX_{\lv} = \{ L(\Bc, \lv) \mid \Bc \in \SC \}$ of $\BU_q^-$ 
associated to $\lv$ as defined in (4.19.1). Note that since $\lv$ is 
standard type, $S_{\Bc_0}^{\lv}$ coincides with $S_{\Bc_0}$. 
For each real root $\b \in \vD^+$, we denote by $f_{\b}^{\lv}$ the 
corresponding root vector. 
Now assume that $\b \in \vD^{(0)}_{<}$ or $\b \in \vD^{(0)}_{>}$.
We define a root vector $f^{\Bh'}_{\b}$ by a similar formula as (4.9.1) for 
the infinite reduced word $\Bh'$ given in (6.2.1). 
(But note that here we consider the total order $\vD^{(0)}_{>} \lv \d \lv \vD^{(0)}_{<}$,
which is not a one-row order on $\vD^{\min}_+$.)  

\begin{lem}  
For $\b \in \vD^{(0)}_{>}$ or $\b \in \vD^{(0)}_{<}$, we have 
$f_{\b}^{\lv} = f^{\Bh'}_{\b}$. 
\end{lem}  

\begin{proof}
Assume that $\b \in \vD^{(0)}_{>}$.  Let $\b_0$ be the minimal element for 
the order $\lv$.  Since $\b$ is finitely far from $\b_0$, the order $\lv$ 
is given by $\b_0 \lv \cdots \lv \b_N = \b \lv \cdots$ for 
$\b_0, \dots, \b_N \in \vD^{(0)}_{>}$.  Thus one can find a finite set $M$ 
containing $\b_0, \dots, \b_N$ and containing $\g, \g'$ such that $\g + \g' = \b_i$
for all $i$. We can find a one-row convex order whose restriction on $M$ coincides with
$\lv$, and define $f_{\b}^{\lv}$ by using this one-row order.  Thus 
$f_{\b}^{\lv}$ coincides with $f^{\Bh'}_{\b}$.  The case where $\b \in \vD^{(0)}_{<}$ is
similar.    
\end{proof}

\para{6.8}
We consider the action of $\s$ on root vectors.  
Recall that $\s$ preserves $\vD^{(0)}_{>}$ and $\vD^{(0)}_{<}$, respectively. 
Thus $\s$ permutes $\vD^{(1)}_{>} \scup \vD^{(1)}_{<}$, but 
$\vD^{(1)}_{>}, \vD^{(1)}_{<}$ are not necessarily $\s$-stable.  
In fact, in the case where $\ve = 2$, $\s$ permutes the sets 
$\vD^{(1)}_{>}$ and $\vD^{(1)}_{<}$. We have the following result. 

\begin{cor}  
Assume that $\b \in \vD^{(0)}_{>}$ or $\b \in \vD^{(0)}_{<}$.
Then $\s(f_{\b}^{\lv}) = f_{\s(\b)}^{\lv}$. 
\end{cor}
\begin{proof}
Take $\b \in \vD^{(0)}_{>}$. 
Since $\vD^{(0)}_{>}$ is $\s$-invariant, $\s(\b) \in \vD^{(0)}_{>}$. 
Let $\Bh'$ be as in (6.2.1) obtained from $\ul\Bh$. .  
Then $\s(\Bh')$ also
corresponds to $\ul\Bh$, and $\s(\Bh')$ is obtained from $\Bh'$ by the permutations 
of each factor $\eta_k$ in (6.2.1).  We see that $\s(f^{\Bh'}_{\b}) = f^{\Bh'}_{\s(\b)}$ 
for $\b \in \vD^{(0)}_{>}$. This holds also for $\b \in \vD^{(0)}_{<}$.  
The corollary then follows from Lemma 6.8.
\end{proof}

\para{6.10.}
We consider the action of $\s$ on $S_{\Bc_0}$.  Since $I_0$ is not $\s$-stable, 
$S_{\Bc_0}$ is not $\s$-stable.  But if we consider the subset $I_0'$ of $I_0$, 
then $I_0'$ is $\s$-stable.  
It is clear from the discussion in 1.5 that 
$\s(\wt P_{i, k}) = \wt P_{\s(i), k}$ for $i \in I_0'$, where 
$\wt P_{i,k}$ is defined in (1.5.7).  Note that since $\BU_q^-$ is untwisted, 
$d_i = 1$ in this case. 
Recall the $I_0$-tuple of partitions 
$\Bc_0 =(\r^{(i)})_{i \in I_0}$ as in 1.5.  We regard 
the $I_0'$-tuple of partitions $\Bc_0 = (\r^{(i)})_{i \in I_0'}$ as 
$I_0$-tuple of partitions  
by putting $\r^{(i)}$ empty partition for $i \notin I_0'$. 
Then $\s$ acts on the set of $I_0'$-tuple of partitions by
$\s(\Bc_0) = \Bc_0'$, where $\Bc_0' = (\r^{(\s(i))})_{i \in I_0'}$. 
The following is immediate from the definition.

\begin{lem}  
Assume that $\Bc_0$ is an $I_0'$-tuple of partitions.
Then we have $\s(S_{\Bc_0}) = S_{\Bc_0'}$, where $\Bc_0' = \s(\Bc_0)$ 
is defined as above.  
\end{lem}   

\para{6.12.}
We now consider the action of $\s$ on the PBW-basis $\SX_{\lv}$. 
Although $\s$ does not preserve $\SX_{\lv}$, we can determine 
the subset $\SX^{\s}_{\lv}$ of $\s$-fixed elements in $\SX_{\lv}$. 
Take $\Bc = (\Bc_+, \Bc_0, \Bc_-) \in \SC$ with $\Bc_+ = (c_{\b})_{\b \lv \d}$, 
$\Bc_- = (c_{\g})_{\d \lv \g}$.  
Let $\SC_+$ be the set of $\Bc_+ $ such that $c_{\b} = 0$ 
for $\b \in \vD^{(1)}_{>}$, and $\SC_-$ the set of $\Bc_- $ 
such that $c_{\g} = 0$ for $\g \in \vD^{(1)}_{<}$.  
Then $\s$ acts on $\SC_+$, and on $\SC_-$ as permutations, and we denote 
by $\SC_+^{\s}, \SC_-^{\s}$ the subset of $\s$-fixed elements.  We also
denote by $\SC_0$ the set of $I_0'$-tuple of partitions as in 6.10, 
and $\SC_0^{\s}$ the set of $\s$-fixed elements in $\SC_0$.  We define 
a subset $\SC^{\s}$ as the set of $\Bc = (\Bc_+, \Bc_0, \Bc_-) \in \SC$ 
such that $\Bc_+ \in \SC_+^{\s}, \Bc_0 \in \SC_0^{\s}, \Bc_- \in \SC_-^{\s}$. 
\par
We also consider the set 
$\ul\SX_{\ul\Bh} = \{ L(\ul\Bc, \ul 0) \mid \ul\Bc \in \ul\SC\}$
of PBW-basis of $\ul\BU_q^-$ (the case where $\ul p = \ul 0$).  Then as in 
[SZ], we have a natural bijection $\SC^{\s} \simeq \ul\SC$. 
We have the following result.

\begin{thm}  
$\SX^{\s}_{\lv} = \{ L(\Bc, \lv) \mid \Bc \in \SC^{\s} \}$. 
\end{thm} 

\begin{proof}
By Corollary 6.9 and Lemma 6.11, we see that $\{L(\Bc, \lv) \mid \Bc \in \SC^{\s}\}$ 
is contained in $\SX_{\lv}^{\s}$. We show the equality holds. 
Let $b(\Bc, \lv)$ the canonical basis associated to $L(\Bc, \lv)$. 
We first note that 
\par\medskip\noindent
(6.13.1) \ If $L(\Bc, \lv)$ is $\s$-invariant, then $b(\Bc, \lv)$ is 
also $\s$-invariant.
\par\medskip
In fact, assume that $L(\Bc, \lv)$ is $\s$-invariant. One can write as 
\begin{equation*}
b(\Bc, \lv) = L(\Bc, \lv) + \sum_{\Bc < \Bc'}a_{\Bc, \Bc'}b(\Bc',\lv)
\end{equation*}
with $a_{\Bc, \Bc'} \in q\BZ[q]$. Then we have
\begin{equation*}
\s(b(\Bc, \lv)) = L(\Bc, \lv) + \sum_{\Bc < \Bc'}a_{\Bc, \Bc'}\s(b(\Bc',\lv)). 
\end{equation*}
This implies that $b(\Bc, \lv) - \s(b(\Bc, \lv))$ 
is a linear combination of $b(\Bc', \lv) - \s(b(\Bc', \lv))$ with coefficients 
$a_{\Bc,\Bc'} \in q\BZ[q]$. Since $\s$ commutes with the bar-involution, 
this implies that $a_{\Bc, \Bc'} = 0$ for any $\Bc'$.  Thus (6.13.1) holds. 
\par
For $\nu \in Q^{\s}$ we consider the subset $\wt\bB^{\s}(\nu)$ of $\wt\bB$ 
as in 2.11.  We can also define a subset $\SX_{\lv}^{\s}(\nu)$ of $\SX_{\lv}$ 
consisting of $\s$-fixed elements with weight $\nu$.  Then (6.13.1) implies that 
\begin{equation*}
\tag{6.13.2}
|\SX^{\s}_{\lv}(\nu)| \le |\wt\bB^{\s}(\nu)|/2.
\end{equation*}
Let  
$\ul\SX_{\ul\Bh}(\ul\nu)$ be the subset of $\ul\SX_{\ul\Bh}$ consisting of 
elements with weight $\ul\nu \in \ul Q$. 
It is known  that $|\ul\SX_{\ul\Bh}(\ul\nu)| = \dim (\ul\BU_q^-)_{\ul\nu}$. 
By Proposition 2.12, under the identification $Q^{\s} = \ul Q, \nu \lra \ul\nu$, 
we have $|\wt\bB^{\s}(\nu)|/2 = \dim (\ul\BU_q^-)_{\ul\nu}$.  Hence 
\begin{equation*}
\tag{6.13.3}
|\ul\SX_{\ul\Bh}(\ul\nu)| = |\wt\bB^{\s}(\nu)|/2 
\end{equation*}
On the other hand, by a similar discussion in [SZ], we see that 
\begin{equation*}
\tag{6.13.4}
\sharp\{ L(\Bc, \lv) \in (\BU_q^-)_{\nu} \mid \Bc \in \SC^{\s} \} = 
        |\ul\SX_{\ul\Bh}(\ul\nu)|.
\end{equation*}
Combining (6.13.2), (6.13.3) and (6.13.4), we have, for any $\nu \in Q^{\s}$,  
\begin{equation*}
\tag{6.13.5}
\sharp\{ L(\Bc, \lv) \in (\BU_q^-)_{\nu} \mid \Bc \in \SC^{\s} \} = 
                  |\SX^{\s}_{\lv}(\nu)|. 
\end{equation*}
The theorem follows from (6.13.5). 
\end{proof}

\para{6.14.}
Each $L(\Bc, \lv) \in \SX^{\s}_{\lv}$ gives an element in ${}_{\BA'}\BU_q^{-,\s}$, 
hence we can consider the image $\pi(L(\Bc, \lv)) \in V_q$. 
On the other hand, for any $L(\ul\Bc, \ul 0) \in \ul\SX_{\ul\Bh}$, we can consider the image 
$\Phi(L(\ul\Bc, \ul 0)) \in V_q$.  The following result is a generalization of 
[SZ, Theorem 2.10, (i)], and is proved in a similar way by using [SZ, Lemma 1.13].

\begin{thm}  
$\Phi(L(\ul\Bc, \ul 0)) = \pi(L(\Bc, \lv))$ under the correspondence 
$\ul\SC \simeq \SC^{\s}, \ul\Bc \lra \Bc$. 
\end{thm} 

By making use of the parametrization of canonical basis in terms of the
PBW-basis, the bijection $\wt\bB^{\s} \simeq \wt{\ul \bB}$ can be 
described explicitly as follows.

\begin{thm}  
We have $\wt\bB^{\s} = \{ \pm b(\Bc, \lv) \mid \Bc \in \SC^{\s} \}$ 
and $\wt{\ul\bB} = \{ \pm b(\ul\Bc, \ul 0) \mid \ul\Bc \in \ul\SC \}$.  
The correspondence $\f : \wt\bB^{\s} \isom  \wt{\ul\bB}$ in Theorem 3.18 
is given by the bijection $\SC^{\s} \isom \ul\SC, \Bc \mapsto \ul\Bc$. 
\end{thm} 

The following result is a generalization of [SZ, Theorem 0.10], 
and is proved by using Theorem 2.4, Theorem 6.15 and Theorem 6.16.
(Note that in [SZ], the expansion of canonical bases in terms of 
PBW-bases was considered. However, such a formula does not hold 
in the general setting since $\s$ permutes $\wt\bB$, 
but does not permute $\SX_{\lv}$.) 

\begin{cor}  
Let $L(\ul\Bc, \ul 0)$ be the PBW-basis of $\ul\BU_q^-$, and 
$L(\Bc, \lv)$ be the corresponding $\s$-stable PBW-basis of $\BU_q^-$ with 
$\Bc \in \SC^{\s}, \ul\Bc \in \ul\SC$ under the bijection in Theorem 6.15.
We write them as 

\begin{align*}
L(\ul\Bc, \ul 0) &= b(\ul\Bc, \ul 0) + \sum_{\ul \Bc < \ul\Bd}a_{\ul\Bd}
                                             b(\ul\Bd, \ul 0), \\
L(\Bc, \lv)      &= b(\Bc, \lv) + \sum_{\Bc < \Bd'}a'_{\Bd'}b(\Bd', \lv). 
\end{align*}
If $b(\Bd',\lv)$ is the $\s$-stable canonical basis of $\BU_q^-$ corresponding to
$\ul\Bd$, then we have

\begin{equation*}
a'_{\Bd'} \equiv a_{\ul\Bd} \pmod \ve.
\end{equation*}

\end{cor}
\par
\par\vspace{1.5cm}
\noindent
T. Shoji \\
School of Mathematica Sciences, Tongji University \\ 
1239 Siping Road, Shanghai 200092, P.R. China  \\
E-mail: \verb|shoji@tongji.edu.cn|

\par\vspace{0.5cm}
\noindent
Z. Zhou \\
School of Mathematical Sciences, Tongji University \\ 
1239 Siping Road, Shanghai 200092, P.R. China  \\
E-mail: \verb|forza2p2h0u@163.com|

\end{document}